\documentclass[11pt]{amsart}
\title[The liftability question for stable equivalences]{The liftability question for stable equivalences \\ between representation-finite self-injective algebras}
\author{Nengqun Li and Yuming Liu*}

\address{Nengqun Li and Yuming Liu
\newline School of Mathematical Sciences
\newline Laboratory of Mathematics and Complex Systems
\newline Beijing Normal University
\newline Beijing 100875
\newline P.R.China}
\email{ymliu@bnu.edu.cn}
\email{wd0843@163.com}

\date{version of \today}
\usepackage{amsmath}
\usepackage{amssymb}
\usepackage{amsxtra}
\usepackage{abstract}
\usepackage{mathrsfs}
\usepackage{mathtools}
\usepackage[all]{xy}
\usepackage{amscd}
\usepackage{amsthm}
\usepackage[dvips]{graphicx}
\usepackage{ulem}
\usepackage{color}
\usepackage{appendix}

\newtheorem{Thm}{Theorem}[section]
\newtheorem{Lem}[Thm]{Lemma}
\newtheorem{Def}[Thm]{Definition}
\newtheorem{Cor}[Thm]{Corollary}
\newtheorem{Prop}[Thm]{Proposition}
\newtheorem{Ex1}[Thm]{Example}
\newtheorem{Rem1}[Thm]{Remark}

\newenvironment{Rem}{\begin{Rem1}\rm}{\end{Rem1}}

\newcommand{\lra}{\longrightarrow}

\newcommand{\ra}{\rightarrow}
\newcommand{\sdp}{\times\kern-.2em\vrule height1.1ex depth-.05ex}
\newcommand{\epi}{\lra \kern-.8em\ra}

\newcommand{\rad}{\mbox{rad}}
\newcommand{\soc}{\mbox{soc}}

\newcommand{\al}{\alpha}
\newcommand{\be}{\beta}

 \normalbaselines
\begin{document}
\renewcommand{\thefootnote}{\alph{footnote}}
\setcounter{footnote}{-1} \footnote{* Corresponding author.}
\setcounter{footnote}{-1} \footnote{\it{Mathematics Subject
Classification(2020)}: 16G10, 18G65.}
\renewcommand{\thefootnote}{\alph{footnote}}
\setcounter{footnote}{-1} \footnote{ \it{Keywords}: liftability question, (nonstandard) RFS algebra,
stable equivalence (of Morita type), (stable) Picard group, (standard) derived equivalence.}

\maketitle

\begin{abstract}
Let $k$ be an algebraically closed field. It is known that any stable equivalence between standard representation-finite self-injective $k$-algebras (without block of Loewy length 2) lifts to a standard derived equivalence, in particular, it is of Morita type. We show that the same holds for any stable equivalence between nonstandard representation-finite self-injective $k$-algebras. We also fill a gap in the original proof in standard case. This gives a complete solution of the liftability question raised by H. Asashiba about twenty years ago.
\end{abstract}	

\bigskip
\section{Introduction}
Throughout this paper, we fix an algebraically closed field $k$. Unless otherwise stated, all algebras will be
finite-dimensional $k$-algebras, and all their modules will be finite-dimensional left modules. For an algebra $A$, we denote by $\mathrm{mod}A$ the category of $A$-modules, and by $\underline{\mathrm{mod}}A$ the stable category of $\mathrm{mod}A$ modulo projective modules. We abbreviate (indecomposable, basic) representation-finite self-injective algebra over $k$ (not isomorphic to the underlying field $k$) by RFS algebra.

The classification of RFS algebras was finished in the 1980's by Riedtmann and her collaborators using covering theory and the notion of (combinatorial) configurations. Let $Q$ be a Dynkin quiver of type $A_n, D_n, E_6, E_7$ or $E_8$, and let $\mathbb{Z}Q$ be the translation quiver associated to $Q$ with the translation denoted as $\tau$. For a translation quiver $\Gamma$, we let $k\Gamma$ be its path category, whose objects are the vertices of $\Gamma$ and morphisms are generated by the paths of $\Gamma$ over $k$; and let $k(\Gamma)$ be the mesh category of $\Gamma$, which is a factor category of $k\Gamma$ by the mesh ideal. Riedtmann showed in \cite{Riedtmann1} that for an RFS algebra $A$, the stable AR-quiver $\prescript{}{s}{\Gamma}_{A}$ is of the form $\mathbb{Z}Q/\Pi$, where $Q$ (the underlying graph of which is called the tree class of $A$) is a Dynkin quiver of type $A_n, D_n, E_6, E_7$ or $E_8$, and $\Pi$ is some admissible subgroup of the automorphism group of $\mathbb{Z}Q$.

\begin{Def}\label{combinatorial-configuration} {\rm(\cite{Riedtmann2})}
Let $\Delta$ be a stable translation quiver. A (combinatorial)
configuration $\mathcal{C}$ is a set of vertices of $\Delta$ which
satisfy the following conditions:
\begin{enumerate}
\item For any $e, f\in \mathcal{C}$,
$Hom_{k(\Delta)}(e,f)= \left\{\begin{array}{ll} 0 & (e\neq f), \\
k & (e=f).\end{array}\right.$
\item For any $e\in\Delta_0$, there exists some $f\in \mathcal{C}$ such that $Hom_{k(\Delta)}(e,f)\neq 0$.
\end{enumerate}
\end{Def}

In \cite{Riedtmann2,Riedtmann4,BLR}, it was shown that the
isoclasses of $\Pi$-stable $\mathbb{Z}Q$ configurations (two
configurations $\mathcal{C}$ and $\mathcal{C}'$ of $\mathbb{Z}Q$ are
called isomorphic if $\mathcal{C}$ is mapped onto $\mathcal{C}'$
under an automorphism of $\mathbb{Z}Q$) correspond bijectively to the isoclasses of
RFS algebras of tree class $Q$ with admissible group $\Pi$, except in
the case of $Q=D_{3m}$ with underlying field having characteristic
$2$. In such a case, an isoclass of $\Pi$-stable $\mathbb{Z}Q$ configuration might correspond to two
isoclasses of RFS algebras; both are symmetric algebras, one of
which is standard, while the other one is nonstandard. Here, an RFS algebra $A$ is called standard if
$k(\Gamma_A)$ is equivalent to $\mathrm{ind}A$, where $\Gamma_A$ is the
AR-quiver of $A$ and $\mathrm{ind}A$ is the full subcategory of $\mathrm{mod}A$ whose
objects are specific representatives of the isoclasses of
indecomposable $A$-modules. Nonstandard RFS algebras are RFS algebras which are not standard. We will introduce the representative algebra of nonstandard RFS algebras in next section.

The derived and stable classifications of RFS algebras were given by Asashiba in 1999. Now we briefly recall his results. First we need
to define the type of an RFS algebra $A$.  If $A$ is as above, by a theorem of Riedtmann \cite{Riedtmann1}, $\Pi$ has the form $\langle \zeta\tau^{-r}\rangle$ where $\zeta$ is some automorphism of $Q$ and $\tau$ is the translation. We also recall the Coxeter numbers of $Q=A_n,D_n,E_6,E_7,E_8$ are $h_Q = n+1,2n-2,12,18,30$ respectively. The frequency of $A$ is defined to be $f_A = r/(h_Q-1)$ and the torsion order $t_A$ of $A$ is defined as the order of $\zeta$.  The type of $A$ is defined as the triple $(Q,f_A,t_A)$ and denoted by $\mathrm{typ}(A)$.

\begin{Thm}\label{dclassRFS} {\rm(\cite{Asashiba1999})} Let $A$
and $B$ be RFS $k$-algebras for $k$ algebraically closed.
\begin{enumerate}
\item If $A$ is standard and $B$ is non-standard, then $A$ and $B$ are not stably equivalent, and hence not derived equivalent.
\item If both $A$ and $B$ are standard, or both non-standard, the
following are equivalent:
\begin{enumerate}
\item $A,B$ are derived equivalent;
\item $A,B$ are stably equivalent of Morita type;
\item $A,B$ are stably equivalent;
\item $A,B$ have the same stable AR-quiver;
\item $A,B$ have the same type.
\end{enumerate}
\item The types of standard RFS algebras are the following:
\begin{enumerate}
\item $\{ (A_n,s/n,1) | n,s\in \mathbb{N}\}$,
\item $\{ (A_{2p+1},s,2) | p,s\in \mathbb{N}\}$,
\item $\{ (D_n,s,1) | n,s\in \mathbb{N}, n\geq 4\}$,
\item $\{ (D_{3m},s/3,1) | m,s\in \mathbb{N} , m\geq 2, 3\nmid s\}$,
\item $\{ (D_n,s,2) | n,s\in \mathbb{N}, n\geq 4\}$,
\item $\{ (D_4,s,3) | s\in \mathbb{N}\}$,
\item $\{ (E_n,s,1) | n=6,7,8; s\in \mathbb{N}\}$,
\item $\{ (E_6,s,2) | s\in \mathbb{N}\}$.
\end{enumerate}
Non-standard RFS algebras are of type $(D_{3m},1/3,1)$ for some
$m\geq 2$.
\end{enumerate}
\end{Thm}

An interesting question arising from the above classification theorem is the following:

\medskip
{\it The liftability question {\rm(\cite{Asashiba2003})}}: Is every stable equivalence $\phi:\underline{\mathrm{mod}}A\rightarrow \underline{\mathrm{mod}}B$ between two RFS $k$-algebras $A$ and $B$ lifts to a standard derived equivalence? In particular, is $\phi$ a stable equivalence of Morita type?

\medskip
Asashiba answered positively the above question for most standard RFS algebras, and the other few cases in standard case were solved by Dugas \cite{Dugas2014} using mutation theory (see also \cite{CKL} for an alternative proof).

\begin{Rem} \label{counterexample} (1) We noticed that there are counterexamples of \cite[Proposition 3.3]{Asashiba2003} if $\Lambda$ and $\Pi$ have Loewy length 2. Let $\Lambda=\Pi=A$ be the RFS algebra given by the quiver $$
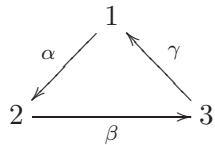
\xymatrix{
			& 1\ar[ld]_{\alpha} &  \\
			2 \ar[rr]_{\beta} &   & 3\ar[lu]_{\gamma}
		}$$ and relations $\beta\alpha=\gamma\beta=\alpha\gamma=0$. Let $\phi:\underline{\mathrm{mod}}A\rightarrow \underline{\mathrm{mod}}A$ be the stable equivalence given by $\phi(1)=2$, $\phi(2)=1$, $\phi(3)=3$. Since the configuration of $A$ is the set of simple modules, $\phi$ preserves the configuration of $A$. But $\phi$ does not commute with the loop functor $\Omega_{A}$, hence is not a stable equivalence of Morita type.

The reason why such counterexamples appear is that in the proof of \cite[Proposition 3.3]{Asashiba2003}, the author assumed that each stable equivalence between standard RFS algebras induces a translation quiver isomorphism between the corresponding stable AR-quivers, thus the proposition needs an additional assumption that the stable equivalence in question commutes with AR-translation up to isomorphisms. However, by \cite[Chapter X, Corollary 1.9(2)]{ARS}, it might be wrong for self-injective algebras with blocks of Loewy length 2.

(2) One key step in the proof of \cite[Proposition 3.3]{Asashiba2003} is to construct a functor $\Phi:k(\Gamma_{\Lambda})\rightarrow k(\Gamma_{\Pi})$ from an equivalence functor $\phi': k(\prescript{}{s}{\Gamma}_{\Lambda})\rightarrow k(\prescript{}{s}{\Gamma}_{\Pi})$ such that $\phi'(\mathcal{C}_{\Lambda})=\mathcal{C}_{\Pi}$, where $\mathcal{C}_{\Lambda}$ (resp. $\mathcal{C}_{\Pi}$) corresponds to the radicals of indecomposable projective $\Lambda$-modules (resp. the radicals of indecomposable projective $\Pi$-modules). Since $k(\prescript{}{s}{\Gamma}_{\Lambda})$ (resp. $k(\prescript{}{s}{\Gamma}_{\Pi})$) is a quotient category of $k({\Gamma}_{\Lambda})$ (resp. $k({\Gamma}_{\Pi})$), and the values of $\Phi$ on the arrows $\alpha$ of $\prescript{}{s}{\Gamma}_{\Lambda}$ depend on the choices of the lifting of $\phi'(\alpha)$ in $k(\Gamma_{\Pi})$, one needs to choose carefully these values so that $\Phi$ preserves all mesh relations. It seems that this verification is skipped in the proof of \cite[Proposition 3.3]{Asashiba2003} and it is not clear for us how to fill this gap under the assumption on $\phi'$.

To obtain a corrected form of \cite[Proposition 3.3]{Asashiba2003}, we need to strengthen the condition on $\phi'$ so that $\phi': k(\prescript{}{s}{\Gamma}_{A})\rightarrow k(\prescript{}{s}{\Gamma}_{A})$ is an isomorphism functor inducing identity map on the objects of $k(\prescript{}{s}{\Gamma}_{A})$, where $A$ is some representative algebra for a given type of standard RFS algebras. Then we can construct a functor $\Phi:k(\Gamma_{A})\rightarrow k(\Gamma_{A})$ which lifts $\phi'$ under the above condition. The construction of $\Phi$ is based on several technical lemmas presented in Section 3. The detailed explanation will be given in Appendix A.
\end{Rem}

However, the liftability question in nonstandard case remains open. The main purpose of the present paper is to give a positive answer in nonstandard case.

\begin{Thm}\label{main-result}
Let $A$, $B$ be nonstandard RFS algebras. Then each stable equivalence $\phi:\underline{\mathrm{mod}}A\rightarrow \underline{\mathrm{mod}}B$ lifts to a standard derived equivalence. In particular, it is of Morita type.
\end{Thm}

Thus, together with Appendix A, we give a complete solution of the liftability question in \cite{Asashiba2003} with a corrected form (Proposition \ref{corrected-form}) of \cite[Proposition 3.3]{Asashiba2003}.

\medskip
This article is organized as follows. In Section 2, we recall the representative algebra $\Lambda$ of nonstandard RFS algebras, its stable AR-quiver $\prescript{}{s}{\Gamma}_{\Lambda}$ and $\underline{\mathrm{ind}}\Lambda$ in terms of a quotient category of the path category $k\prescript{}{s}{\Gamma}_{\Lambda}$. In Section 3, we prove a technical result (Proposition \ref{Prop, lift}) on lifting of a stable auto-equivalence of the nonstandard RFS algebra $\Lambda$ to a Morita equivalence. The last section is devoted to prove our main result Theorem \ref{main-result}. As a by-product, we determine the stable Picard group $\mathrm{StPic}(\Lambda)$ of $\Lambda$ (Proposition \ref{Prop,stable-picard}).

In Appendix A, we prove a corrected form (Proposition \ref{corrected-form}) of \cite[Proposition 3.3]{Asashiba2003} and explain how to use it to reprove \cite[Theorem 3.1]{Asashiba2003}. In Appendix B, we give a detailed proof of \cite[Lemma 4.10]{CKL}, which will be used in the proof of our main result.

\section*{Data availability} The datasets generated during the current study are available from the corresponding author on reasonable request.

\section*{Acknowledgements} This research is supported by NSFC (No.12031014). We are very grateful to Professor Hideto Asashiba for valuable comments on our results. We are very grateful to an anonymous referee for careful reading and valuable suggestions on a previous version, which have led to
substantial changes and significant improvement on both mathematics and language of this paper.

\section{The representative algebra of nonstandard RFS algebras}
Let $k$ be an algebraically closed field of characteristic $2$, $\Lambda$ be the representative algebra of nonstandard RFS algebras of type $(D_{3m},1/3,1)$ as in \cite[Appendix 2]{Asashiba2003}, where $m\geq 2$. The algebra $\Lambda$ is given by the quiver $Q$ below with relations $\alpha_{m}\dots \alpha_1=\beta^2$, $\alpha_i\dots\alpha_{i+1}\alpha_i =0$ for all $i\in\{1,...,m\}=\mathbb{Z}/\langle m\rangle$, $\alpha_1\alpha_m=\alpha_1\beta\alpha_m$.

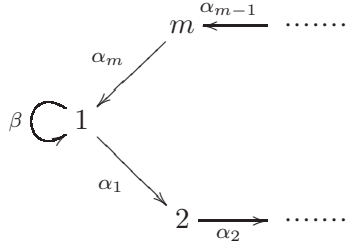
$$
\vcenter{
\xymatrix {
	& m \ar[dl]_{\alpha_m} & \ar[l]_{\alpha_{m-1}} \ar@{{}*{\cdot}{}}[r] &  \\
	1 \ar@(ul,dl)_{\beta} \ar[dr]_{\alpha_1} & & & \\
	& 2 \ar[r]_{\alpha_2} & \ar@{{}*{\cdot}{}}[r]&\\
}}$$

Let $\Gamma_{\Lambda}$ be the AR-quiver of $\Lambda$ and $\prescript{}{s}{\Gamma}_{\Lambda}$ be the stable AR-quiver of $\Lambda$, then $\prescript{}{s}{\Gamma}_{\Lambda}\cong \mathbb{Z}D_{3m}/\langle \tau^{2m-1}\rangle$. We use the following enumeration on the vertices of $D_{3m}$:

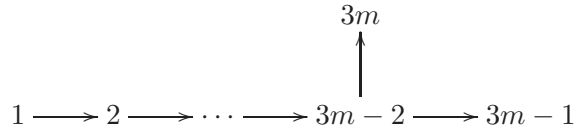
$$\xymatrix{
		& & &3m & \\
		1 \ar[r] & 2\ar[r]  & \cdots  \ar[r] & 3m-2 \ar[u] \ar[r] &3m-1 &
	}$$

Recall that the vertices $3m$ and $3m-1$ are called high vertices of $D_{3m}$ and it is convenient to write a vertex of $\prescript{}{s}{\Gamma}_{\Lambda}$ as its coordinate $(p,q)$, where $p\in\{1,...,2m-1\}=\mathbb{Z}/\langle 2m-1\rangle, 1\leq q\leq 3m$. The simple $\Lambda$-module corresponding to the vertex $i$ ($1\leq i\leq m$) in the quiver $Q$ of $\Lambda$ will be simply denoted by $i$. Note that by \cite[Satz 4.4: 3) a)]{Waschbusch1981}, we can draw the stable AR-quiver $\prescript{}{s}{\Gamma}_{\Lambda}$ so that the simple module $1$ corresponds to $(0,3m)$, the simple module $j$ corresponds to $(2m-j,1)$ for $2\leq j\leq m$. Let $P_i$ be the indecomposable projective $\Lambda$-module corresponding to the vertex $i$. Then we have the following structure of the indecomposable projective $\Lambda$-modules:

\medskip
$P_1=$ 
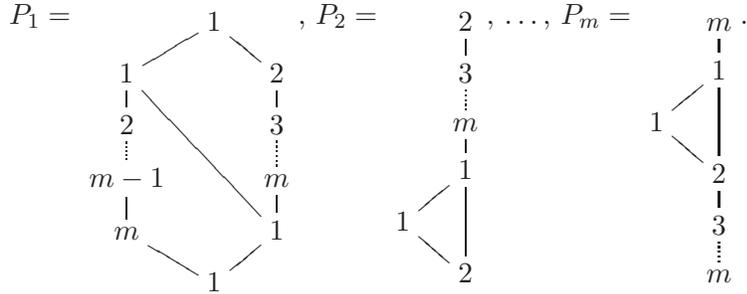
\xymatrix@R=0.5pc@C=0.8pc {&1\ar@{-}[dl]\ar@{-}[dr]& \\
1 \ar@{-}[d]\ar@{-}[dddrr]& & 2\ar@{-}[d] \\
2 \ar@{.}[d] &  & 3 \ar@{.}[d] \\
m-1\ar@{-}[d] & & m\ar@{-}[d] \\
m \ar@{-}[dr]& & 1\ar@{-}[dl] \\
& 1 & }, $P_2=$ \xymatrix@R=0.5pc@C=0.8pc {& 2 \ar@{-}[d]\\
& 3 \ar@{.}[d]\\
& m \ar@{-}[d]\\
& 1 \ar@{-}[dd]\ar@{-}[dl]\\
1 \ar@{-}[dr]& \\
& 2}, $\dots$, $P_m=$ \xymatrix@R=0.5pc@C=0.8pc {& m \ar@{-}[d]\\
& 1 \ar@{-}[dd]\ar@{-}[dl]\\
1 \ar@{-}[dr]& \\
& 2 \ar@{-}[d]\\
&3 \ar@{.}[d] \\
& m}.\\

\noindent Let $\mathcal{C}:=\{{\rad}P_i\mid i=1,2,\cdots,m\}$. Then we have
	
\medskip
${\rad}P_1=$ 
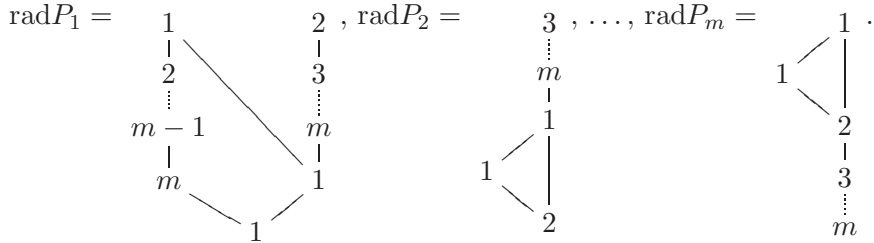
\xymatrix@R=0.5pc@C=0.8pc {
	1 \ar@{-}[d]\ar@{-}[dddrr]& & 2\ar@{-}[d] \\
2 \ar@{.}[d] &  & 3 \ar@{.}[d] \\
m-1\ar@{-}[d] & & m\ar@{-}[d] \\
m \ar@{-}[dr]& & 1\ar@{-}[dl] \\
& 1 & }, ${\rad}P_2=$ \xymatrix@R=0.5pc@C=0.8pc {
& 3 \ar@{.}[d]\\
& m \ar@{-}[d]\\
& 1 \ar@{-}[dd]\ar@{-}[dl]\\
1 \ar@{-}[dr]& \\
& 2
}, $\dots$, ${\rad}P_m=$ \xymatrix@R=0.5pc@C=0.8pc {
& 1 \ar@{-}[dd]\ar@{-}[dl]\\
1 \ar@{-}[dr]& \\
& 2 \ar@{-}[d]\\
&3 \ar@{.}[d] \\
& m}. \\

\noindent The positions of $\mathcal{C}$ in the stable AR-quiver $\prescript{}{s}{\Gamma}_{\Lambda}$ are important, they indicate the positions of the indecomposable projective modules in the AR-quiver $\Gamma_{\Lambda}$. Since the loop functor $\Omega_{\Lambda}$ induces an automorphism of $\prescript{}{s}{\Gamma}_{\Lambda}$ and a bijection between the set of simple $\Lambda$-modules and $\mathcal{C}$, and since $\{(0,3m-1),(2m-1-j,1)\mid j=1,...,m-1\}$ and $\{(0,3m),(2m-1-j,1)\mid j=1,...,m-1\}$ are isomorphic configurations, we can draw the stable AR-quiver $\prescript{}{s}{\Gamma}_{\Lambda}$ in a new way so that $\mathcal{C}$ corresponds to $\{(0,3m-1),(2m-1-j,1)\mid j=1,...,m-1\}$ in $\prescript{}{s}{\Gamma}_{\Lambda}$. In the following, we fix $\mathcal{C}$ to the position $\{(0,3m-1),(2m-1-j,1)\mid j=1,...,m-1\}$ in $\prescript{}{s}{\Gamma}_{\Lambda}$, except in Appendix A (where we fix the simple $\Lambda$-modules $1,2,\cdots,m$ to the positions $(0,3m)$, $(2m-2,1)$, $\cdots$, $(m,1)$, respectively).

Since $\Lambda$ is a nonstandard RFS algebra, according to \cite[Proof of Proposition 3.3]{Riedtmann4}, there is a well-behaved functor $\widetilde{U}: k\Gamma_{\Lambda}\rightarrow \mathrm{ind}\Lambda$ such that it maps each vertex of $\Gamma_{\Lambda}$ to the corresponding indecomposable module and maps each arrow of $\Gamma_{\Lambda}$ to an irreducible morphism, where $k\Gamma_{\Lambda}$ is the path category of the quiver $\Gamma_{\Lambda}$. Moreover, the functor $\widetilde{U}$ induces an isomorphism $U:k\Gamma_{\Lambda}/J\simeq \mathrm{ind}\Lambda$, where $J$ is the ideal of $k\Gamma_{\Lambda}$ generated by the modified mesh relations $\{m_x\mid x\neq (0,3m-1)\}\cup\{m_{(0,3m-1)}+p\}$, where $m_x$ denotes the mesh relation starting at $x$ and $p$ denotes the following path of length $4m$: $(0,3m-1)\rightarrow (1,3m-2)\rightarrow (2,3m-3)\rightarrow (2,3m-2)\rightarrow (3,3m-3)\rightarrow (3,3m-2) \rightarrow\dots\rightarrow(2m-1,3m-2)\rightarrow (2m,3m-3)\rightarrow (2m,3m-2)\rightarrow (2m,3m-1)=(1,3m-1)$. Here is a diagram of the path $p$ in the case $m=3$ (where $\star$ denotes the modules in $\mathcal{C}$ and the path $p$ is marked by the dotted arrows):

$$
\vcenter{
	\xymatrix@R=1.5pc@C=0.6pc {
	&&&&&&&&&&\star\ar@{-->}[dr]&&\bullet\ar[dr]&&\bullet\ar[dr]&&\bullet\ar[dr]\ar[dr]&&\bullet\ar[dr]&&\star\ar[dr]&&\bullet \\
    &&&&&&&&&\bullet\ar[ur]\ar[dr]\ar[r]&\bullet\ar[r]&\bullet\ar[ur]\ar@{-->}[dr]\ar@{-->}[r]&\bullet\ar[r]&\bullet\ar[ur]\ar@{-->}[dr]\ar[r]&\bullet\ar[r]
    &\bullet\ar[ur]\ar@{-->}[dr]\ar[r]
    &\bullet\ar[r]&\bullet\ar[ur]\ar@{-->}[dr]\ar[r]&\bullet\ar[r]&\bullet\ar[ur]\ar@{-->}[dr]\ar[r]&\bullet\ar[r]&\bullet\ar@{-->}[ur] \\	
    &&&&&&&&\bullet\ar[ur]\ar[dr]&&\bullet\ar[ur]\ar[dr]&&\bullet\ar@{-->}[ur]\ar[dr]&&\bullet\ar@{-->}[ur]\ar[dr]&&\bullet\ar@{-->}[ur]\ar[dr]&&\bullet\ar@{-->}[ur]\ar[dr]&&\bullet\ar@{-->}[ur]&& \\
    &&&&&&&\bullet\ar[ur]\ar[dr]&&\bullet\ar[ur]\ar[dr]&&\bullet\ar[ur]\ar[dr]&&\bullet\ar[ur]\ar[dr]&&\bullet\ar[ur]\ar[dr]&&\bullet\ar[ur]\ar[dr]&&\bullet\ar[ur]&&& \\
    &&&&&&\bullet\ar[ur]\ar[dr]&&\bullet\ar[ur]\ar[dr]&&\bullet\ar[ur]\ar[dr]&&\bullet\ar[ur]\ar[dr]&&\bullet\ar[ur]\ar[dr]&&\bullet\ar[ur]\ar[dr]&&\bullet\ar[ur]&&&& \\
    &&&&&\bullet\ar[ur]\ar[dr]&&\bullet\ar[ur]\ar[dr]&&\bullet\ar[ur]\ar[dr]&&\bullet\ar[ur]\ar[dr]&&\bullet\ar[ur]\ar[dr]&&\bullet\ar[ur]\ar[dr]&&\bullet\ar[ur]&&&&& \\
    &&&&\bullet\ar[ur]\ar[dr]&&\bullet\ar[ur]\ar[dr]&&\bullet\ar[ur]\ar[dr]&&\bullet\ar[ur]\ar[dr]&&\bullet\ar[ur]\ar[dr]&&\bullet\ar[ur]\ar[dr]&&\bullet\ar[ur]&&&&&& \\
    &&&\bullet\ar[ur]&&\bullet\ar[ur]&&\bullet\ar[ur]&&\star\ar[ur]&&\star\ar[ur]&&\bullet\ar[ur]&&\bullet\ar[ur]&&&&&&& \\
    &&&0&&1&&2&&3&&4&&0&&1&&&&&
}}$$

Furthermore, $U$ induces an isomorphism $V: k\prescript{}{s}{\Gamma}_{\Lambda}/I\simeq \underline{\mathrm{ind}}\Lambda$, where $I$ is the ideal of $k\prescript{}{s}{\Gamma}_{\Lambda}$ generated by $\{\underline{m_x}\mid x\neq (0,3m-1)\}\cup\{\underline{m_{(0,3m-1)}}+\underline{p}\}$, where $\underline{m_x}$, $\underline{p}$ denote the residue classes of $m_x$, $p$ in $k\prescript{}{s}{\Gamma}_{\Lambda}$ under the natural quotient functor $k\Gamma_{\Lambda}\longrightarrow k\prescript{}{s}{\Gamma}_{\Lambda}$.

We would like to mention an interesting fact on the category $k\prescript{}{s}{\Gamma}_{\Lambda}/I$, although we will not use it in the present paper. It is known that the smallest integer such that each path of length greater than or equal to this integer is zero in $k(\mathbb{Z}D_{3m})$ is $6m-3$ (see \cite[Section 1.1]{BLR}). From the existence of a covering functor $k(\mathbb{Z}D_{3m})\rightarrow k\prescript{}{s}{\Gamma}_{\Lambda}/I$ (see \cite[Section 4]{Riedtmann4} and \cite[Example 3.1c)]{BG}), it is not hard to see that the same holds in the category $k\prescript{}{s}{\Gamma}_{\Lambda}/I$. In particular, ${\rad}(\underline{\mathrm{mod}}\Lambda)$ has nilpotency $6m-3$.

\section{A technical result on stable auto-equivalence \\ of nonstandard RFS algebras}

Let $\mathcal{C}$ be a Krull-Schmidt $k$-additive category. For the definition of the radical ${\rad}(-,-)$ of $\mathcal{C}$ and the irreducible morphisms in $\mathcal{C}$, we refer to \cite[Section 2.2]{Ringel1984}. Recall that if both $X$ and $Y$ are indecomposable, then a morphism $f: X\rightarrow Y$ in $\mathcal{C}$ is irreducible if and only if $f\in {\rad}(X,Y)-{\rad}^2(X,Y)$. We shall frequently use the following simple fact.

\begin{Lem}\label{Lem, lift an irreducible morphism}
If $X$ and $Y$ are two indecomposable nonprojective $A$-modules over a self-injective algebra $A$, and if $f: X\rightarrow Y$ is a morphism in $\mathrm{mod}A$ with the image $\underline{f}$ in $\underline{\mathrm{mod}}A$, then $f$ is an irreducible morphism in $\mathrm{mod}A$ if and only if $\underline{f}$ is an irreducible morphism in $\underline{\mathrm{mod}}A$.
\end{Lem}

The following two lemmas give a way to lift a mesh relation in $\underline{\mathrm{mod}}A$ to a mesh relation in $\mathrm{mod}A$ (that is, an almost split sequence in $\mathrm{mod}A$). For any $A$-module $X$, we denote by $l_{A}(X)$ the composition length of $X$.

\begin{Lem}\label{Lem, lift a mesh relation}
Let $A$ be an algebra, $\Gamma_{A}$ be the AR-quiver of $A$. Let

$$
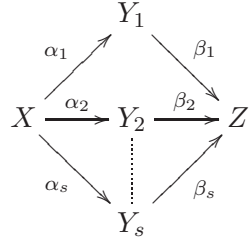
\xymatrix{
& Y_{1}\ar[dr]^{\be_{1}} & \\
X\ar[ur]^{\al_{1}}\ar[r]^{\al_{2}}\ar[dr]_{\al_{s}} & Y_{2}\ar[r]^{\be_{2}}\ar@{.}[d] & Z \\
 &Y_{s}\ar[ur]_{\be_{s}} &
}$$
be a mesh in $\Gamma_{A}$, where $X,Y_{1},\cdots,Y_{s},Z$ are indecomposable nonprojectives. Let $f_{i}$ (resp. $g_{i}$) be irreducible morphisms corresponding to $\al_{i}$ (resp. $\be_{i}$) such that $\sum \underline{g_{i}}\underline{f_{i}}=0$ in $\underline{\mathrm{mod}}A$. Then we have the following.
\begin{enumerate}
\item There exist morphisms $f_{i}'$ for $1\leq i\leq s$ such that $\underline{f_{i}'}=\underline{f_{i}}$ in $\underline{\mathrm{mod}}A$ for $1\leq i\leq s$ and $\sum g_{i}f_{i}'=0$ in $\mathrm{mod}A$.

\item If moreover, $A$ is a self-injective algebra, then there exist morphisms $g_{i}'$ for $1\leq i\leq s$ such that $\underline{g_{i}'}=\underline{g_{i}}$ in $\underline{\mathrm{mod}}A$ for $1\leq i\leq s$ and $\sum g_{i}'f_{i}=0$ in $\mathrm{mod}A$.
\end{enumerate}
\end{Lem}

\begin{proof} The assumption shows that we can assume that $\sum g_{i}f_{i}+vu=0$, where $u:X\rightarrow P$, $v:P\rightarrow Z$, and $P$ a projective module. Since $Y_{1},\cdots,Y_{s}$ are pairwise nonisomorphic,  it is easy to verify that $(g_{1},\cdots,g_{s}):Y_{1}\oplus\cdots\oplus Y_{s}\rightarrow Z$ is irreducible. Since there exists an almost split sequence $0\rightarrow X\rightarrow Y_{1}\oplus\cdots\oplus Y_{s}\rightarrow Z\rightarrow 0$, $l_{A}(Y_{1}\oplus\cdots\oplus Y_{s})>l_{A}(Z)$. Since irreducible morphisms are injective or surjective, we have that $(g_{1},\cdots,g_{s})$ is surjective. Since $P$ is projective, $v$ factors through $(g_{1},\cdots,g_{s})$. Let $v=\sum g_{i}w_{i}$, then $\sum g_{i}(f_{i}+w_{i}u)=0$. Let $f_{i}'=f_{i}+w_{i}u$, we have $\underline{f_{i}'}=\underline{f_{i}}$ in $\underline{\mathrm{mod}}A$ for $1\leq i\leq s$ and $\sum g_{i}f_{i}'=0$ in $\mathrm{mod}A$. This proves (1). Notice that projective modules are also injective over a self-injective algebra, the proof of (2) is dual to that of (1), using the injective envelope of $X$.

\end{proof}

Modify the proof of Lemma \ref{Lem, lift a mesh relation}, we have the following lemma.

\begin{Lem}\label{Lem, lift a mesh relations}
	Let $A$ be an algebra, $\Gamma_{A}$ be the AR-quiver of $A$. Let

	$$
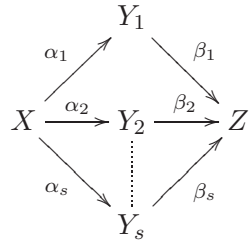
\xymatrix{
		& Y_{1}\ar[dr]^{\be_{1}} & \\
		X\ar[ur]^{\al_{1}}\ar[r]^{\al_{2}}\ar[dr]_{\al_{s}} & Y_{2}\ar[r]^{\be_{2}}\ar@{.}[d] & Z \\
		&Y_{s}\ar[ur]_{\be_{s}} &
	}$$
	be a mesh in $\Gamma_{A}$, where $X,Y_{1},\cdots,Y_{s},Z$ are indecomposable nonprojectives. Let $f_{i}$ (resp. $g_{i}$) be irreducible morphisms corresponding to $\al_{i}$ (resp. $\be_{i}$) such that $\sum \underline{g_{i}}\underline{f_{i}}=0$ in $\underline{\mathrm{mod}}A$. Then we have the following.
\begin{enumerate}	
\item If there exists some $t$ ($1\leq t\leq s$) with $\sum_{i=1}^{t} l_{A}(Y_{i})>l_{A}(Z)$ (or equivalently $(g_{1},\cdots,g_{t})$ is an epimorphism), then there exist morphisms $f_{i}'$ for $1\leq i\leq t$ such that $\underline{f_{i}'}=\underline{f_{i}}$ in $\underline{\mathrm{mod}}A$ for $1\leq i\leq t$ and $\sum_{i=1}^{t} g_{i}f_{i}'+\sum_{i=t+1}^{s} g_{i}f_{i}=0$ in $\mathrm{mod}A$.
	
\item If moreover, $A$ is a self-injective algebra such that there exists some $t$ ($1\leq t\leq s$) with $\sum_{i=1}^{t} l_{A}(Y_{i})>l_{A}(X)$ (or equivalently $(f_{1},\cdots,f_{t})$ is a monomorphism), then there exist morphisms $g_{i}'$ for $1\leq i\leq t$ such that $\underline{g_{i}'}=\underline{g_{i}}$ in $\underline{\mathrm{mod}}A$ for $1\leq i\leq t$ and $\sum_{i=1}^{t} g_{i}'f_{i}+\sum_{i=t+1}^{s} g_{i}f_{i}=0$ in $\mathrm{mod}A$.
\end{enumerate}
\end{Lem}

The following lemma, which is inspired by some idea from the proof of \cite[Proposition 3.3]{Asashiba2003}, gives a sufficient condition for a functor $\psi:\mathrm{mod}A\rightarrow \mathrm{mod}A$ to be an equivalence, where $A$ is an algebra of finite representation type.

\begin{Lem}\label{Lem, equivalence}
Let $A$ be an algebra of finite representation type, $\psi:\mathrm{mod}A\rightarrow \mathrm{mod}A$ be a $k$-functor  such that
\begin{enumerate}
\item $\psi$ preserves the radical of $\mathrm{mod}A$: for every pair $(X,Y)$ of $A$-modules, $$\psi({\rad}_{A}(X,Y))\subseteq {\rad}_{A}(\psi(X),\psi(Y));$$

\item $\psi$ preserves the indecomposability and irreducible morphisms between indecomposables;

\item $\psi$ reflects isomorphism classes: for $A$-modules $X$ and $Y$, if $\psi(X)\cong\psi(Y)$, then $X\cong Y$.
\end{enumerate}
Then $\psi$ is an equivalence.
\end{Lem}

\begin{proof}
First we claim that $\psi({\rad}_{A})+{\rad}^{2}_{A}={\rad}_{A}$. Since $\psi({\rad}_{A})\subseteq {\rad}_{A}$, $\psi({\rad}_{A})+{\rad}^{2}_{A}\subseteq {\rad}_{A}$. To show that ${\rad}_{A} \subseteq \psi({\rad}_{A})+{\rad}^{2}_{A}$, let $X$, $Y$ be indecomposable $A$-modules and $f\in {\rad}_{A}(X,Y)$. If $f\in {\rad}^{2}_{A}$, then $f\in \psi({\rad}_{A})+{\rad}^{2}_{A}$. If $f\notin {\rad}^{2}_{A}$, since $X$, $Y$ are indecomposable, $f$ is irreducible. Since $A$ is of finite representation type, by (2) and (3),  $\psi$ induces a quiver automorphism of $\Gamma_{A}$, where $\Gamma_{A}$ is the AR-quiver of $A$. Then there exists some indecomposable $A$-modules $Z$, $W$ and  an irreducible morphism $g:Z\rightarrow W$ such that $\psi(g)$ is an irreducible morphism from $X$ to $Y$. Since $A$ is of finite representation type, $dim_{k}({\rad}_{A}(X,Y)/{\rad}^{2}_{A}(X,Y))\leq 1$. Then there exists some $\lambda\in k^{*}$ and $h\in {\rad}^{2}_{A}(X,Y)$ such that $f=\lambda\psi(g)+h=\psi(\lambda g)+h$. Since $\lambda g\in {\rad}_{A}(Z,W)$ and $h\in {\rad}^{2}_{A}(X,Y)$, $f\in  \psi({\rad}_{A})+{\rad}^{2}_{A}$.
	
Inductively, we show that $\psi({\rad}^{n}_{A})+{\rad}^{n+1}_{A}={\rad}^{n}_{A}$ for all $n\geq 1$. When $n=1$ it is true. Assume that $\psi({\rad}^{n}_{A})+{\rad}^{n+1}_{A}={\rad}^{n}_{A}$ for some $n\geq 1$, then ${\rad}^{n+1}_{A}=({\rad}^{n}_{A})({\rad}_{A})=(\psi({\rad}^{n}_{A})+{\rad}^{n+1}_{A})(\psi({\rad}_{A})+{\rad}^{2}_{A})= \psi({\rad}^{n+1}_{A})+\psi({\rad}^{n}_{A}){\rad}^{2}_{A}+{\rad}^{n+1}_{A}\psi({\rad}_{A})+{\rad}^{n+3}_{A}$. Since $\psi({\rad}_{A})\subseteq {\rad}_{A}$, $\psi({\rad}^{n}_{A})\subseteq {\rad}^{n}_{A}$. Then $\psi({\rad}^{n}_{A}){\rad}^{2}_{A}+{\rad}^{n+1}_{A}\psi({\rad}_{A})+{\rad}^{n+3}_{A}\subseteq {\rad}^{n+2}_{A}$. Hence ${\rad}^{n+1}_{A}\subseteq \psi({\rad}^{n+1}_{A})+{\rad}^{n+2}_{A}$. Since $\psi({\rad}^{n+1}_{A})\subseteq {\rad}^{n+1}_{A}$ and ${\rad}^{n+2}_{A}\subseteq {\rad}^{n+1}_{A}$, $ \psi({\rad}^{n+1}_{A})+{\rad}^{n+2}_{A}={\rad}^{n+1}_{A}$.
	
We now show that $\psi$ is full. Since $\psi({\rad}^{n}_{A})+{\rad}^{n+1}_{A}={\rad}^{n}_{A}$ for all $n\geq 1$, by induction we have $\psi({\rad}_{A})+{\rad}^{n+1}_{A}={\rad}_{A}$ for all $n\geq 1$. Since $A$ is of finite representation type, ${\rad}_{A}$ is nilpotent. Then $\psi({\rad}_{A})={\rad}_{A}$. For indecomposable $A$-module $X$, $Y$, if $X\ncong Y$, then ${\rad}_{A}(\psi(X),\psi(Y))=Hom_{A}(\psi(X),\psi(Y))$ and $\psi:Hom_{A}(X,Y)\rightarrow Hom_{A}(\psi(X),\psi(Y))$ is an epimorphism. If $X\cong Y$, for $f\in End_{A}(\psi(X))$, $f=\lambda \cdot id+g$ for some $\lambda\in k$ and $g\in {\rad}_{A}(\psi(X),\psi(X))$. Since $\psi({\rad}_{A})={\rad}_{A}$,  $\psi:End_{A}(X)\rightarrow End_{A}(\psi(X))$ is an epimorphism and $\psi:Hom_{A}(X,Y)\rightarrow Hom_{A}(\psi(X),\psi(Y))$ is an epimorphism. Then $\psi$ is full.
	
Since $\psi$ is full, $\psi$ induces an epimorphism
$$\bigoplus\limits_{X,Y\in \mathrm{ind}A}Hom_{A}(X,Y)\rightarrow \bigoplus\limits_{X,Y\in \mathrm{ind}A}Hom_{A}(\psi(X),\psi(Y)).$$
Since $\psi$ induces a quiver automorphism of $\Gamma_{A}$, $\psi$ is dense and \begin{equation*} dim_{k}\bigoplus\limits_{X,Y\in \mathrm{ind}A}Hom_{A}(X,Y)=dim_{k}\bigoplus\limits_{X,Y\in \mathrm{ind}A}Hom_{A}(\psi(X),\psi(Y)). \end{equation*} Hence $\psi$ induces an isomorphism $\bigoplus\limits_{X,Y\in \mathrm{ind}A}Hom_{A}(X,Y)\rightarrow \bigoplus\limits_{X,Y\in \mathrm{ind}A}Hom_{A}(\psi(X),\psi(Y))$. This implies that $\psi$ is full, faithful and dense. Therefore $\psi$ is an equivalence.
	\end{proof}

\begin{Lem}\label{Lem, irreducible morphism}
Let $P$ be an indecomposable projective-injective module over an algebra $A$ with $l_{A}(P)>2$. Let $f:{\rad}P\rightarrow P$, $g:P\rightarrow P/{\soc}P$ be $A$-module homomorphisms with $im(gf)={\rad}P/{\soc}P$. Then both $f$ and $g$ are irreducible morphisms.
\end{Lem}

\begin{proof}
It is sufficient to show that $f$ is injective and $g$ is surjective. If $g$ is not surjective, since ${\rad}P/{\soc}P$ is the unique maximal submodule of $P/{\soc}P$, $im(g)\subseteq {\rad}P/{\soc}P$. Since $im(f)\subseteq {\rad}P$ by the same reason, $im(gf)\subseteq g({\rad}P)\subseteq {\rad}({\rad}P/{\soc}P)\subsetneqq {\rad}P/{\soc}P$, a contradiction. Then $g$ is surjective.

Since $im(f)\subseteq {\rad}P$, we have a sequence of morphisms ${\rad}P\xrightarrow[]{f'} {\rad}P\xrightarrow[]{g'} {\rad}P/{\soc}P$, where $f'$, $g'$ are morphisms induced from $f$, $g$ respectively. Since $l_{A}(P)>2$, ${\rad}^{2}(P)\neq 0$. Then ${\soc}P\subseteq {\rad}^{2}(P)$. Since $g$ is surjective, $g'$ is surjective and $ker(g')={\soc}P\subseteq {\rad}^{2}(P)={\rad}({\rad}P)$. Then $g'$ is an essential epimorphism. Since $im(gf)={\rad}P/{\soc}P$, $g'f'$ is an epimorphism. Then $f'$ is an epimorphism, which implies that $f'$ an isomorphism. Therefore $f$ is injective.
\end{proof}

Recall from Section 2 that for the representative algebra $\Lambda$ of nonstandard RFS algebras of type $(D_{3m},1/3,1)$, there is an isomorphism $U:k\Gamma_{\Lambda}/J\rightarrow \mathrm{ind}\Lambda$ such that $U$ maps each vertex of $\Gamma_{\Lambda}$ to the corresponding indecomposable module and maps each arrow of $\Gamma_{\Lambda}$ to an irreducible morphism. Let $q: k\Gamma_{\Lambda}/J\rightarrow k\prescript{}{s}{\Gamma}_{\Lambda}/I$ and $q: \mathrm{ind}\Lambda\rightarrow \underline{\mathrm{ind}}\Lambda$ be the natural quotient functors, respectively. Since $U$ restricts to an isomorphism between the subcategory of projective vertices of $k\Gamma_{\Lambda}/J$ and the subcategory of projective modules in $\mathrm{ind}\Lambda$, $U$ induces naturally an isomorphism $V:k\prescript{}{s}{\Gamma}_{\Lambda}/I\rightarrow\underline{\mathrm{ind}}\Lambda$, which also maps each arrow of $\prescript{}{s}{\Gamma}_{\Lambda}$ to an irreducible morphism. Let $\phi': \underline{\mathrm{ind}}\Lambda\rightarrow \underline{\mathrm{ind}}\Lambda$ be an isomorphism, which induces an isomorphism $\phi'_{0}: k\prescript{}{s}{\Gamma}_{\Lambda}/I\rightarrow k\prescript{}{s}{\Gamma}_{\Lambda}/I$. Then we have the following diagram with three commutative faces:
	
	$$\xymatrix@R=1.3pc{
	k\Gamma_{\Lambda}/J \ar[rd]^{U}_{\sim} \ar[dd]_{q}& & k\Gamma_{\Lambda}/J\ar[rd]^{U}_{\sim}\ar@{.>}[dd]_{q} & \\
	& \mathrm{ind}\Lambda\ar[dd]_(0.3){q} & & \mathrm{ind}\Lambda \ar[dd]_{q}\\
	k\prescript{}{s}{\Gamma}_{\Lambda}/I \ar[rd]^{V}_{\sim}\ar@{.>}[rr]_(0.7){\phi'_{0}} & & k\prescript{}{s}{\Gamma}_{\Lambda}/I \ar@{.>}[rd]^{V}_{\sim}& \\
	& \underline{\mathrm{ind}}\Lambda\ar[rr]_{\phi'} & & \underline{\mathrm{ind}}\Lambda. \\
}$$
$$\mbox{Figure }1$$

\medskip
We often use the following fact without mentioning.

\begin{Lem}\label{Lem, preserve irreducible morphism}
Under the above assumptions, let $\alpha$ be an arrow of $\Gamma_{\Lambda}$ which is also considered as a morphism in $k\Gamma_{\Lambda}/J$, and let $\sigma(\alpha)$ be a
morphism in $k\Gamma_{\Lambda}/J$ such that $q(\sigma(\alpha)) = \phi'_{0}(q(\alpha))$. Then $U(\sigma(\alpha))$ is an irreducible morphism
in $\mathrm{ind}\Lambda$.
\end{Lem}

\begin{proof} To show that $U(\sigma(\alpha))$ is an irreducible morphism
in $\mathrm{ind}\Lambda$, by Lemma \ref{Lem, lift an irreducible morphism}, it suffices to show that $qU(\sigma(\alpha))$ is an irreducible morphism
in $\underline{\mathrm{ind}}\Lambda$. Since $U:k\Gamma_{\Lambda}/J\rightarrow \mathrm{ind}\Lambda$ maps each arrow of $\Gamma_{\Lambda}$ to an irreducible morphism,  $U(\al)$ is an irreducible morphism, and $qU(\al)$ is an irreducible morphism in $\underline{\mathrm{ind}}\Lambda$. Since $\phi'$ is an isomorphism and
 \begin{equation*} qU(\sigma(\al))=Vq(\sigma(\al))=V\phi'_{0}q(\al)=\phi'Vq(\al)=\phi'qU(\al),
\end{equation*} $qU(\sigma(\alpha))$ is an irreducible morphism
in $\underline{\mathrm{ind}}\Lambda$.

\end{proof}

It is clear that the similar result as Lemma \ref{Lem, preserve irreducible morphism} holds for any RFS algebra $A$.

\medskip
We now prove the main result of this section, which is a corrected form of \cite[Proposition 3.3]{Asashiba2003} in nonstandard case.

\begin{Prop}\label{Prop, lift}  Let $k$ be an algebraically closed field of characteristic $2$, $\Lambda$ be the representative algebra of nonstandard RFS algebras of type $(D_{3m},1/3,1)$, where $m\geq 2$. Let $\phi:\underline{\mathrm{mod}}\Lambda\rightarrow\underline{\mathrm{mod}}\Lambda$ be a stable equivalence such that $\phi (X)\cong X$ for any $X\in\underline{\mathrm{mod}}\Lambda$. Then $\phi$ lifts to a Morita equivalence.
\end{Prop}

{\small\xymatrix@R=1.7pc@C=0.6pc@W=0mm{
 & & & & & & & & & & & & & & & & & \star\ar[dr] & & \bullet\ar[dr]|{\quad\be_{3m-2}^{0}} & & \bullet & & & & \\
 & & & & & & & & & & & & & & & & \bullet\ar[dr]\ar[ur]\ar[r] & \bullet\ar[r]^{\delta^{1}} & \bullet\ar[dr]|{\quad\be_{3m-3}^{0}}\ar[ur]|{\al_{3m-2}^{0}\quad}\ar[r]^{\gamma^{0}} & \bullet\ar[r]^{\delta^{0}} & \bullet\ar[ur] &  & & & & \\
 & & & & & & & & & & & & & & & & & \bullet\ar[dr]|{\be_{3m-4}^{0}}\ar[ur]|{\al_{3m-3}^{0}\quad}\ar@{{}*{\cdot}{}}[ddll] & & \bullet\ar[ur] & & & & & & \\
 & & & & & & & & & & & & & & & & & & \bullet\ar[ur]\ar@{{}*{\cdot}{}}[ddll] & & &  & & & & \\
 & & & & & & & & & & & & & & & \bullet\ar[dr]^{\be_{2m}^{0}} & & & & & &  & & & & \\
 & & & & & & & & & & & & & & \bullet\ar[dr]|{\be_{2m-1}^{0}}\ar[ur]^{\al_{2m}^{0}} & & \bullet & & & & &  & & & & \\
 & & & & & & & & & & & & & \bullet\ar[dr]_{\al_{2m-2}^{0}}\ar[ur]^{\al_{2m-1}^{0}} & & \bullet\ar[dr]\ar[ur] & & & & & &  & & & & \\
 & & & & & & & & & & & & & & \bullet\ar[dr]_{\al_{2m-3}^{0}}\ar[ur]|{\be_{2m-2}^{0}} & & \bullet\ar@{{}*{\cdot}{}}[ddrr] & & & & &  & & & & \\
 & & & & & & & & & & & & & & & \bullet\ar[ur]|{\be_{2m-3}^{0}}\ar@{{}*{\cdot}{}}[ddrr] & & & & & &  & & & & \\
 & & & & & & & & & & & & & & & & & & \bullet\ar[dr] & & &  & & & & \\
 & & & & & & & & & & & & & & & & & \bullet\ar[dr]_{\al_{m+1}^{0}}\ar[ur]|{\be_{m+2}^{0}} & & \bullet\ar[dr] & & & & & & \\
 & & & & & & & & & & & & & & & & & & \bullet\ar[dr]|{\be_{m}^{0}}\ar[ur]|{\be_{m+1}^{0}} & & \bullet & & & & & \\
 & & & & & & & & & & & & & & & & & \bullet\ar[dr]_{\al_{m-1}^{0}}\ar[ur]^{\al_{m}^{0}} & & \bullet\ar[dr]\ar[ur] & & & & & & \\
 & & & & & & & & & & & & & & & & & & \bullet\ar[dr]_{\al_{m-2}^{0}}\ar[ur]|{\be_{m-1}^{0}} & & \bullet\ar@{{}*{\cdot}{}}[dr] & & & & & \\
 & & & & & & & & & & & & & & & & & & & \bullet\ar[ur]|{\be_{m-2}^{0}}\ar@{{}*{\cdot}{}}[dr] & &\bullet\ar[dr] & & & & \\
 & & & & & & & & & & & & & & & & & & & &\bullet\ar[ur]|{\be_{2}^{0}}\ar[dr]_{\al_{1}^{0}} & & \bullet\ar[dr] & & & \\
 &\bullet\ar@{{}*{\cdot}{}}"2,17" & &\bullet\ar@{{}*{\cdot}{}}"7,14" & &\bullet\ar@{{}*{\cdot}{}}"8,15" & & & & & & & & & &\star\ar@{{}*{\cdot}{}}"14,19" & & &\cdots & & & \star\ar[ur]|{\be_{1}^{0}} & & \bullet & & \\
  &0 & &1 &  &2\ar@{{}*{\cdot}{}}"18,16" & & & & & & & & &  &m\ar@{{}*{\cdot}{}}"18,22" & & & & & &2m-2 & &2m-1 &  & \\
 }}
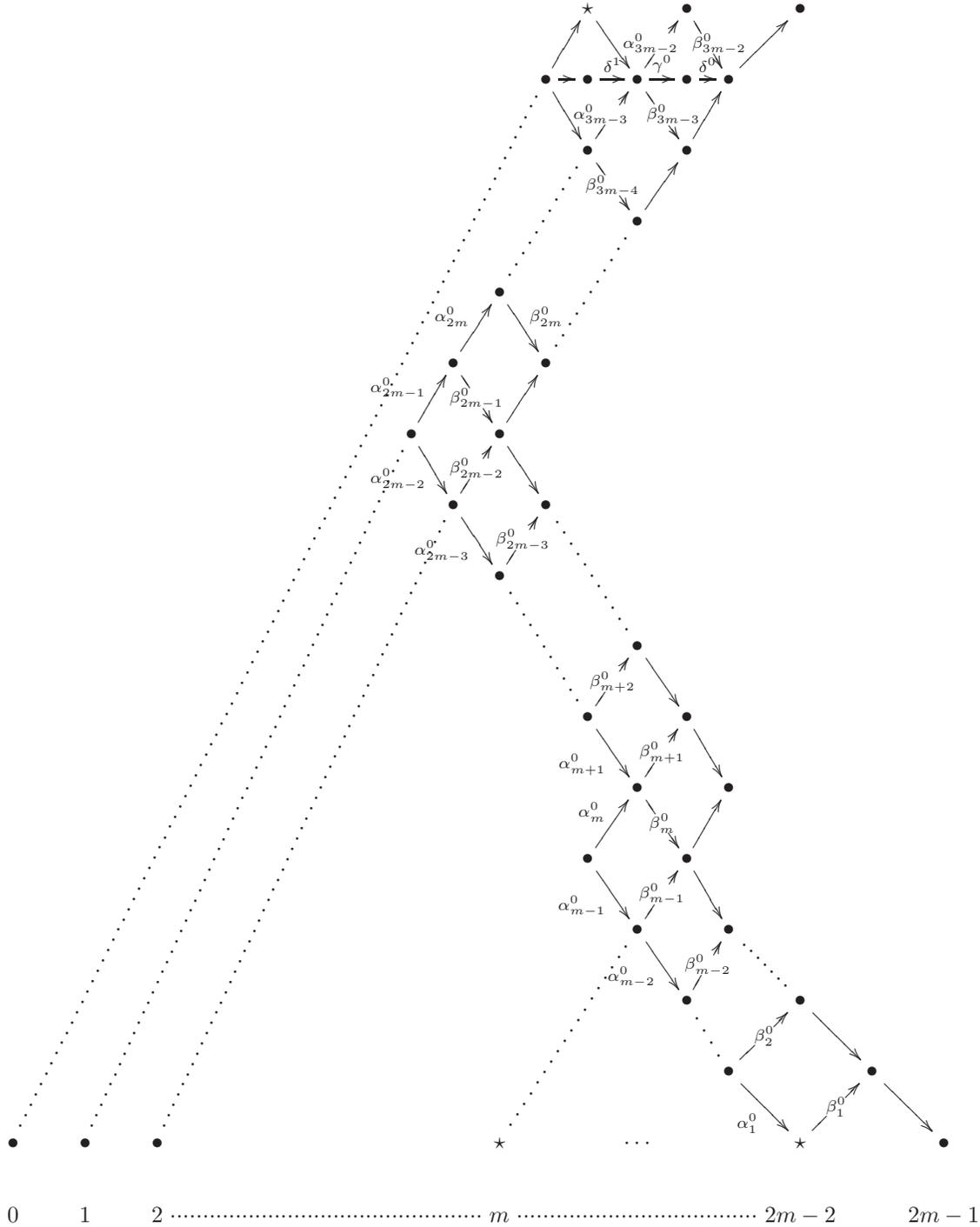
$$\mbox{Figure }2\mbox{ (part of }\prescript{}{s}{\Gamma}_{\Lambda}\mbox{ where }\star \mbox{ denotes the modules in }\mathcal{C},\mbox{ and }0\mbox{ and }2m-1\mbox{ are identified})$$

\begin{proof} Since $\phi:\underline{\mathrm{mod}}\Lambda\rightarrow\underline{\mathrm{mod}}\Lambda$ is a stable equivalence, it induces an equivalence $\phi':\underline{\mathrm{ind}}\Lambda\rightarrow\underline{\mathrm{ind}}\Lambda$. Then we have the diagram in Figure 1. Note that since the categories $k\Gamma_{\Lambda}/J$, $\mathrm{ind}\Lambda,k\prescript{}{s}{\Gamma}_{\Lambda}/I$, and $\underline{\mathrm{ind}}\Lambda$ are basic, equivalences between them automatically become isomorphisms. By construction, the left, the right and the bottom faces of the diagram in Figure 1 are (strict) commutative.

We will divide our proof into two steps. In Step 1, we define a functor $\psi'_{0}:k\Gamma_{\Lambda}/J\rightarrow k\Gamma_{\Lambda}/J$ such that the back face of the diagram in Figure 1 is commutative; the functor $\psi'_{0}$ then induces a functor $\psi':\mathrm{ind}\Lambda\rightarrow \mathrm{ind}\Lambda$ such that the top and the front faces of the diagram in Figure 1 are also commutative. In Step 2, we show that $\psi':\mathrm{ind}\Lambda\rightarrow \mathrm{ind}\Lambda$ is an isomorphism.

\medskip
{\it Step 1: To define a functor $\psi'_{0}: k\Gamma_{\Lambda}/J\rightarrow k\Gamma_{\Lambda}/J$ which lifts the functor $\phi'_{0}:k\prescript{}{s}{\Gamma}_{\Lambda}/I\rightarrow k\prescript{}{s}{\Gamma}_{\Lambda}/I$.}

We begin to define $\psi'_{0}$ on objects as the identity map.
For each arrow $\alpha:x \rightarrow y$ of $\Gamma_{\Lambda}$ (by abuse of notation, we also denote by $\alpha$ the corresponding morphism in $k\Gamma_{\Lambda}$ or in $k\Gamma_{\Lambda}/J$), we need to choose a morphism $\sigma(\alpha):x\rightarrow y$ of $k{\Gamma}_{\Lambda}/J$ such that $q\sigma(\alpha)=\phi'_{0}q(\alpha)$ in $k\prescript{}{s}{\Gamma}_{\Lambda}/I$ and all $\sigma(\alpha)$ are compatible with the modified mesh relations in $k\Gamma_{\Lambda}/J$. Once this is done, we get the desired functor $\psi'_{0}: k\Gamma_{\Lambda}/J\rightarrow k\Gamma_{\Lambda}/J$. We will divide Step 1 into four substeps.

{\it Step 1.1: To choose a section $D_{3m}'$ in $\prescript{}{s}{\Gamma}_{\Lambda} \cong \mathbb{Z}D_{3m}/\langle \tau^{2m-1}\rangle$ (note that $\prescript{}{s}{\Gamma}_{\Lambda} \cong \mathbb{Z}D_{3m}'/\langle \tau^{2m-1}\rangle$ since the underlying graphs of $D_{3m}'$ and $D_{3m}$ are isomorphic) and to define inductively (to the left direction) the values of $\sigma$ on all the arrows in $\prescript{}{s}{\Gamma}_{\Lambda}$ except for the arrows $\be^{0}_{i}$ $(1\leq i \leq 3m-2)$ and $\delta^{0}$ in Figure 2, such that $\sigma(m_x)=0$  in $k\Gamma_{\Lambda}/J$ for all vertices $x$ in $\prescript{}{s}{\Gamma}_{\Lambda}$ and not in $\mathcal{C}$ once the values of $\sigma$ on all arrows which belong to $m_{x}$ have been defined.}

Recall that we fix $\mathcal{C}=\{{\rad}P_i\mid i=1,2,\cdots,m\}$ to the position $\{(0,3m-1)$, $(2m-2,1)$, $\dots$, $(m,1)\}$ in the stable AR-quiver $\prescript{}{s}{\Gamma}_{\Lambda}$ (cf. Section 2). The arrows in the section $D_{3m}'$ are marked by $\al_{1}^{0}, \cdots, \al_{3m-2}^{0},\gamma^{0}$ from the bottom to the top (see Figure 2). Note that the irreducible morphisms corresponding to $\alpha^{0}_{i}$ are monomorphisms for $1\leq i\leq m-2$ or $2m-1\leq i\leq 3m-2$ and are epimorphisms for $m-1 \leq i\leq 2m-2$, and that the irreducible morphism corresponding to $\gamma^{0}$ is an epimorphism. We just give one example to show that $\alpha^{0}_{2m-1}$ corresponds to an irreducible monomorphism. Since there exists an almost split sequence $0\rightarrow U(2m-1,1)\rightarrow U(2m-1,2)\rightarrow U(2m,1)\rightarrow 0$, $l_{\Lambda}(U(2m-1,1))<l_{\Lambda}(U(2m-1,2))$. Inductively, since there exists almost split sequences $0\rightarrow U(2m-1-i,1+i)\rightarrow U(2m-1-i,2+i)\oplus U(2m-i,i)\rightarrow U(2m-i,1+i)\rightarrow 0$ for each $1\leq i\leq 2m-2$, we have that $l_{\Lambda}(U(1,2m-1))<l_{\Lambda}(U(1,2m))$. Therefore $\alpha^{0}_{2m-1}$ is a monomorphism. For example, in the case $m=4$, the section $D_{12}'$ is marked by the dotted arrows in $\prescript{}{s}{\Gamma}_{\Lambda}$ as follows.

$$
\vcenter{
	\small\xymatrix@R=1.2pc@C=0.6pc {
	&&&&&&&&&&\star\ar[dr]&&\bullet\ar[dr]&&\bullet\ar[dr]&&\bullet\ar[dr]\ar[dr]&&\bullet\ar[dr]&&\bullet\ar[dr]&&\bullet\ar[dr]&&\bullet \\
    &&&&&&&&&\bullet\ar[ur]\ar[dr]\ar[r]&\bullet\ar[r]&\bullet\ar@{^{(}-->}[ur]^{\alpha^{0}_{10}}\ar[dr]\ar@{-->>}[r]^{\gamma^{0}}&\bullet\ar[r]&\bullet\ar[ur]\ar[dr]\ar[r]&\bullet\ar[r]
    &\bullet\ar[ur]\ar[dr]\ar[r]
    &\bullet\ar[r]&\bullet\ar[ur]\ar[dr]\ar[r]&\bullet\ar[r]&\bullet\ar[ur]\ar[dr]\ar[r]&\bullet\ar[r]&\bullet\ar[ur]\ar[dr]\ar[r]&\bullet\ar[r]&\bullet\ar[ur]& \\	
    &&&&&&&&\bullet\ar[ur]\ar[dr]&&\bullet\ar@{^{(}-->}[ur]^{\alpha^{0}_{9}}\ar[dr]&&\bullet\ar[ur]\ar[dr]&&\bullet\ar[ur]\ar[dr]&&\bullet\ar[ur]\ar[dr]&&\bullet\ar[ur]\ar[dr]&&\bullet\ar[ur]\ar[dr]&&\bullet\ar[ur]&& \\
    &&&&&&&\bullet\ar[ur]\ar[dr]&&\bullet\ar@{^{(}-->}[ur]^{\alpha^{0}_{8}}\ar[dr]&&\bullet\ar[ur]\ar[dr]&&\bullet\ar[ur]\ar[dr]&&\bullet\ar[ur]\ar[dr]&&\bullet\ar[ur]\ar[dr]&&\bullet\ar[ur]\ar[dr]&&\bullet\ar[ur]&&& \\
    &&&&&&\bullet\ar[ur]\ar[dr]&&\bullet\ar@{^{(}-->}[ur]^{\alpha^{0}_{7}}\ar@{-->>}[dr]^{\alpha^{0}_{6}}&&\bullet\ar[ur]\ar[dr]&&\bullet\ar[ur]\ar[dr]&&\bullet\ar[ur]\ar[dr]&&\bullet\ar[ur]\ar[dr]&&\bullet\ar[ur]\ar[dr]&&\bullet\ar[ur]&&&& \\
    &&&&&\bullet\ar[ur]\ar[dr]&&\bullet\ar[ur]\ar[dr]&&\bullet\ar[ur]\ar@{-->>}[dr]^{\alpha^{0}_{5}}&&\bullet\ar[ur]\ar[dr]&&\bullet\ar[ur]\ar[dr]&&\bullet\ar[ur]\ar[dr]&&\bullet\ar[ur]\ar[dr]&&\bullet\ar[ur]&&&&& \\
    &&&&\bullet\ar[ur]\ar[dr]&&\bullet\ar[ur]\ar[dr]&&\bullet\ar[ur]\ar[dr]&&\bullet\ar[ur]\ar[dr]&&\bullet\ar[ur]\ar[dr]&&\bullet\ar[ur]\ar[dr]&&\bullet\ar[ur]\ar[dr]&&\bullet\ar[ur]&&&&&& \\
    &&&\bullet\ar[ur]\ar[dr]&&\bullet\ar[ur]\ar[dr]&&\bullet\ar[ur]\ar[dr]&&\bullet\ar@{-->>}[ur]^{\alpha^{0}_{4}}\ar@{-->>}[dr]^{\alpha^{0}_{3}}&&\bullet\ar[ur]\ar[dr]&&\bullet\ar[ur]\ar[dr]&&\bullet\ar[ur]\ar[dr]&&\bullet\ar[ur]&&&&&&& \\
    &&\bullet\ar[ur]\ar[dr]&&\bullet\ar[ur]\ar[dr]&&\bullet\ar[ur]\ar[dr]&&\bullet\ar[ur]\ar[dr]&&\bullet\ar[ur]\ar@{^{(}-->}[dr]^{\alpha^{0}_{2}}&&\bullet\ar[ur]\ar[dr]&&\bullet\ar[ur]\ar[dr]&&\bullet\ar[ur]&&&&&&&& \\
    &\bullet\ar[ur]\ar[dr]&&\bullet\ar[ur]\ar[dr]&&\bullet\ar[ur]\ar[dr]&&\bullet\ar[ur]\ar[dr]&&\bullet\ar[ur]\ar[dr]&&\bullet\ar[ur]\ar@{^{(}-->}[dr]^{\alpha^{0}_{1}}&&\bullet\ar[ur]\ar[dr]&&\bullet\ar[ur]&&&&&&&&& \\
    \bullet\ar[ur]&&\bullet\ar[ur]&&\bullet\ar[ur]&&\bullet\ar[ur]&&\star\ar[ur]&&\star\ar[ur]&&\star\ar[ur]&&\bullet\ar[ur]&&&&&&&&&& \\
    0&&1&&2&&3&&4&&5&&6&&7&&&&&&&&
}}$$
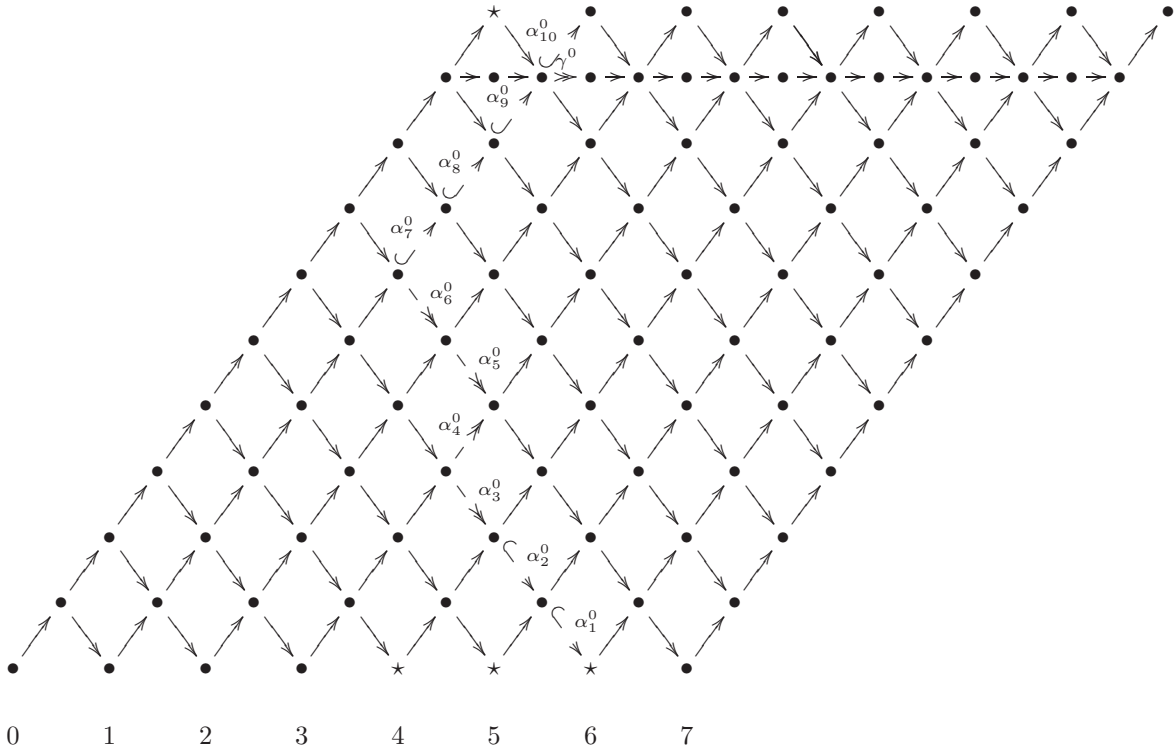
$$\mbox{Figure 3 (where }\star\mbox{ denotes the modules in }\mathcal{C},\mbox{ and 0 and 7 are identified})$$

\medskip
For each $1\leq i\leq 3m-2$, suppose $\al_{i}^{0}$ be the arrow from $x$  to $y$, define $\be_{i}^{0}$ be the arrow from $y$ to $\tau^{-1}x$ in $\Gamma_{\Lambda}$. Let $\al_{i}^{r}=\tau^r\al_{i}^{0}$, $\be_{i}^{r}=\tau^r\be_{i}^{0}$, $\gamma^{r}:(1-r,3m-2)\rightarrow(1-r,3m)$, $\delta^{r}:(1-r,3m)\rightarrow(2-r,3m-2)$.

For each arrow $\al^{0}_{i}$ ($1\leq i\leq 3m-2$) on the section $D_{3m}'$, choose $\sigma(\al^{0}_{i})$ to be a morphism of $k\Gamma_{\Lambda}/J$ such that $q(\sigma(\al^{0}_{i}))=\phi'_{0}(q(\al^{0}_{i}))$, choose $\sigma(\gamma^{0})$ to be a morphism of $k\Gamma_{\Lambda}/J$ such that $q(\sigma(\gamma^{0}))=\phi'_{0}(q(\gamma^{0}))$, and choose a morphism  $\sigma(\be^{1}_{1})$ of $k\Gamma_{\Lambda}/J$ such that $q(\sigma(\be^{1}_{1}))=\phi'_{0}(q(\be^{1}_{1}))$. Next, choose temporarily morphisms $\sigma'(\be^{1}_{2})$, $\sigma'(\al^{1}_{1})$ of $k\Gamma_{\Lambda}/J$ such that $q(\sigma'(\be^{1}_{2}))=\phi'_{0}(q(\be^{1}_{2}))$, $q(\sigma'(\al^{1}_{1}))=\phi'_{0}(q(\al^{1}_{1}))$ (such morphisms $\sigma'(-)$ may not be compatible with the modified mesh
relations but we will adjust them to get our promised morphisms). By Lemma \ref{Lem, preserve irreducible morphism}, $U(\sigma(\al^{0}_{2}))$, $U(\sigma'(\be^{1}_{2}))$, $U(\sigma(\be^{1}_{1}))$, $U(\sigma'(\al^{1}_{1}))$ are irreducible in $\mathrm{ind}\Lambda$. On the other hand, we have
\begin{equation*}
qU(\sigma(\al^{0}_{2})\sigma'(\be^{1}_{2})+\sigma(\be^{1}_{1})\sigma'(\al^{1}_{1}))=Vq(\sigma(\al^{0}_{2})
\sigma'(\be^{1}_{2})+\sigma(\be^{1}_{1})\sigma'(\al^{1}_{1}))=V\phi'_{0}q(\al^{0}_{2}\be^{1}_{2}+\be^{1}_{1}\al^{1}_{1})=0
\end{equation*}
in $\underline{\mathrm{ind}}\Lambda$, since $\al^{0}_{2}\be^{1}_{2}+\be^{1}_{1}\al^{1}_{1}$ lies in the modified mesh ideal $J$.
 By Lemma \ref{Lem, lift a mesh relation}(1), there exist morphisms $f$, $g$ in $\mathrm{mod}\Lambda$ such that $qU(\sigma'(\be^{1}_{2}))=q(f)$ and $qU(\sigma'(\al^{1}_{1}))=q(g)$ in $\underline{\mathrm{mod}}\Lambda$ and $U(\sigma(\al^{0}_{2}))f+U(\sigma(\be^{1}_{1}))g=0$ in $\mathrm{mod}\Lambda$.
Define $\sigma(\be^{1}_{2}):=U^{-1}(f)$ and $\sigma(\al^{1}_{1}):=U^{-1}(g)$ (note that $U$ is an isomorphism). Now we have $q(\sigma(\al^{1}_{1}))=\phi'_{0}(q(\al^{1}_{1}))$ and $q(\sigma(\be^{1}_{2}))=\phi'_{0}(q(\be^{1}_{2}))$ in $k\prescript{}{s}{\Gamma}_{\Lambda}/I$, and $\sigma(\al^{0}_{2})\sigma(\be^{1}_{2})+\sigma(\be^{1}_{1})\sigma(\al^{1}_{1})=0$ in $k{\Gamma}_{\Lambda}/J$. Similarly one can define $\sigma(\al^{1}_{i})$, $\sigma(\be^{1}_{i})$ for all $1\leq i\leq 3m-2$ and $\sigma(\gamma^1)$, $\sigma(\delta^1)$ such that $q\sigma=\phi'_{0}q$ for all these arrows, and that $\sigma(m_x)=0$ in $k\Gamma_{\Lambda}/J$ for the vertices $x$ in $\prescript{}{s}{\Gamma}_{\Lambda}\setminus\mathcal{C}$ once the values of $\sigma$ on all arrows that belong to $m_{x}$ have been defined.

By induction, one can define the values of $\sigma$ on $\al^{r}_{i}$, $\be^{t}_{i}$, $\gamma^r$, $\delta^t$ for all $1\leq i\leq 3m-2$, $0\leq r\leq 2m-2$, $1\leq t\leq 2m-2$ such that $q\sigma=\phi'_{0}q$ on all these arrows, and that $\sigma(m_x)=0$ in $k\Gamma_{\Lambda}/J$ for all vertices $x$ in $\prescript{}{s}{\Gamma}_{\Lambda}$ and not in $\mathcal{C}$ once the values of $\sigma$ on all arrows which belong to $m_{x}$ have been defined.

\medskip
{\it Step 1.2: To define the values of $\sigma$ on arrows $\be^{0}_{i}$ $(1\leq i \leq 3m-2)$ and $\delta^{0}$ in Figure 2 such that $\sigma(m_x)=0$ in $k\Gamma_{\Lambda}/J$ for all vertices $x$ in $\prescript{}{s}{\Gamma}_{\Lambda}$ except for the vertex $(1,3m-2)$ and the vertices in $\mathcal{C}$}.

We start from the middle ones ($\be^{0}_{m}$ and $\be^{0}_{m-1}$) to define the values of $\sigma$ on arrows $\be^{0}_{i}$ $(1\leq i \leq 3m-2)$. By Lemma \ref{Lem, lift a mesh relation}(2), one can choose morphisms $\sigma(\be^{0}_{m})$ and $\sigma(\be^{0}_{m-1})$ such that $q(\sigma(\be^{0}_{i}))=\phi'_{0}(q(\be^{0}_{i}))$ for $i=m-1,m$ and $\sigma(\be^{0}_{m-1})\sigma(\al^{0}_{m-1})+\sigma(\be^{0}_{m})\sigma(\al^{0}_{m})=0$. Note that since the irreducible morphisms corresponding to $\alpha^{0}_{i}$ are monomorphisms for $1\leq i\leq m-2$ or $2m-1\leq i\leq 3m-3$, and the irreducible morphisms corresponding to $\alpha^{2m-2}_{i}=\tau^{-1}\al_{i}^{0}$ are epimorphisms for $m+1\leq i\leq 2m-2$, using Lemma \ref{Lem, lift a mesh relations}, the values of $\sigma$ on arrows $\beta^{0}_{m-2},\cdots,\beta^{0}_{1}$ and arrows $\beta^{0}_{m+1},\cdots,\beta^{0}_{3m-3}$ can be defined inductively such that $q\sigma=\phi'_{0}q$ on all these arrows, and that $\sigma(m_{x})=0$ in $k\Gamma_{\Lambda}/J$ for all vertices $x$ once the values of $\sigma$ on all arrows which belong to $m_{x}$ have been defined. Finally, using Lemma \ref{Lem, lift a mesh relation}(1), the value of $\sigma$ on $\be^{0}_{3m-2}$, $\delta^{0}$ can be defined, which satisfy $q\sigma=\phi'_{0}q$ on arrows $\be^{0}_{3m-2}$, $\delta^{0}$, and $\sigma(m_{x})=0$ in $k\Gamma_{\Lambda}/J$ for $x=(1,3m-1)$ or $(1,3m)$. Therefore the values of $\sigma$ on all arrows of  $\prescript{}{s}{\Gamma}_{\Lambda}$ have been defined, which satisfy $q\sigma=\phi'_{0}q$ on all arrows of $\prescript{}{s}{\Gamma}_{\Lambda}$ and $\sigma(m_x)=0$ in $k\Gamma_{\Lambda}/J$ for all vertices $x$ in $\prescript{}{s}{\Gamma}_{\Lambda}$ except the vertex $(1,3m-2)$ and the vertices which correspond to the radical of some indecomposable projective module.

\medskip
{\it Step 1.3: To adjust the values of $\sigma$ on arrows $\al^{0}_{3m-2}$ and $\gamma^0$ in Figure 2 such that $\sigma(m_x)=0$ in $k\Gamma_{\Lambda}/J$ for all vertices $x$ in $\prescript{}{s}{\Gamma}_{\Lambda}$ and not in $\mathcal{C}$}.

Since there exists an exact sequence \begin{equation*} 0\rightarrow U(1,1)\rightarrow U(1,2)\rightarrow U(2,1)\rightarrow 0, \end{equation*}  $l_{\Lambda}(U(1,2))>l_{\Lambda}(U(2,1))$. Since there exists exact sequences \begin{equation*} 0\rightarrow U(1,i)\rightarrow U(1,i+1)\oplus U(2,i-1)\rightarrow U(2,i)\rightarrow 0  \end{equation*} for $2\leq i\leq 3m-3$, by induction $l_{\Lambda}(U(1,3m-2))>l_{\Lambda}(U(2,3m-3))$. Since there exists an exact sequence \begin{equation*}
 0\rightarrow U(1,3m-2)\rightarrow U(1,3m-1)\oplus U(1,3m)\oplus U(2,3m-3)\rightarrow U(2,3m-2)\rightarrow 0,
 \end{equation*} $l_{\Lambda}(U(1,3m-1))+l_{\Lambda}(U(1,3m))>l_{\Lambda}(U(2,3m-2))$. By Lemma \ref{Lem, lift a mesh relations}(1), the values of $\sigma$ on $\al^{0}_{3m-2}$ and $\gamma^0$ can be changed such that it satisfies $q(\sigma(\al^{0}_{3m-2}))=\phi'_{0}(q(\al^{0}_{3m-2}))$, $q(\sigma(\gamma^0))=\phi'_{0}(q(\gamma^0))$ and $\sigma(\be^{0}_{3m-2})\sigma(\al^{0}_{3m-2})+\sigma(\delta^0)\sigma(\gamma^0)+\sigma(\al^{2m-2}_{3m-3})\sigma(\be^{0}_{3m-3})=0$. As a result, we get $\sigma(m_x)=0$ for the vertex $x=(1,3m-2)$. Note that the above adjustment changes the value of $\sigma$ on $\gamma^0$ and we still need to show that $\sigma(m_x)=0$ for the vertex $x=(0,3m)$.

Indeed, we will show that there is no nonzero morphism from $U(0,3m)$ to $U(1,3m)$ which factors though a projective module, which implies $U(\sigma(\gamma^0))U(\sigma(\delta^{1}))=0$.

Note that $U(1,3m-2)={\rad}P_{1}/{\soc}P_{1}=$ \xymatrix@R=0.5pc@C=0.8pc {
	1 \ar@{-}[d]\ar@{-}[dddrr]& & 2\ar@{-}[d] \\
	2 \ar@{.}[d] &  & 3 \ar@{.}[d] \\
	m-1\ar@{-}[d] & & m\ar@{-}[d] \\
	m & & 1 \\
	 }  \\
 Let $X=$ \xymatrix@R=0.5pc@C=0.8pc {
 	 2\ar@{-}[d] \\
  3 \ar@{.}[d] \\
 	 m\ar@{-}[d] \\
  1 \\
 },
   $Y=$ \xymatrix@R=0.5pc@C=0.8pc {
	1 \ar@{-}[d] \\
	2 \ar@{.}[d]  \\
	m-1\ar@{-}[d]  \\
	m   \\
} .
There is a nonsplit exact sequence $0\rightarrow X\rightarrow {\rad}P_{1}/{\soc}P_{1}\rightarrow Y\rightarrow 0$. Since $\Lambda$ is symmetric, $DTr=\Omega^2$ and $\Omega^2(Y)=\Omega(Z)=X$, where $Z=$ \xymatrix@R=0.5pc@C=0.8pc {
	1 \ar@{-}[d]\ar@{-}[dddrr]& & \\
	2 \ar@{.}[d] &  &  \\
	m-1\ar@{-}[d] & &  \\
	m \ar@{-}[dr]& & 1\ar@{-}[dl] \\
	& 1 & } \\
Since each nonzero endomorphism of $Y$ is an isomorphism, each nonisomorphism $h:Y\rightarrow Y$ is zero, hence $h$ factors through ${\rad}P_{1}/{\soc}P_{1}\rightarrow Y$. By \cite[Chapter V, Proposition 2.2]{ARS}, $0\rightarrow X\rightarrow {\rad}P_{1}/{\soc}P_{1}\rightarrow Y\rightarrow 0$ is an almost split sequence.  Then $U(0,3m)\cong X$ and $U(1,3m)\cong Y$. If a morphism $f:X\rightarrow Y$ factors through the projective cover $P_{1}$ of $Y$, let $f=vu$, $u:X\rightarrow P_{1}$, $v:P_{1}\rightarrow Y$, we have $u=\lambda \iota$, $v=\mu\pi$, where $\lambda,\mu\in k$, $\iota:X\rightarrow P_{1}$ be the inclusion and $\pi:P_{1}\rightarrow Y$ be the projection. Since $\pi\iota=0$, $f=0$. It implies that each morphism form $X$ to $Y$ which factors through a projective module is equal to zero. Since $q(\sigma(\delta^{1}))=\phi'_{0}(q(\delta^{1}))$ and $q(\sigma(\gamma^0))=\phi'_{0}(q(\gamma^0))$, $q(\sigma(\gamma^0))q(\sigma(\delta^{1}))=0$, $U(\sigma(\gamma^0))U(\sigma(\delta^{1})):X\rightarrow Y$ factors through a projective module. Hence $U(\sigma(\gamma^0))U(\sigma(\delta^{1}))=0$ and $\sigma(\gamma^0)\sigma(\delta^{1})=0$.

\medskip
{\it Step 1.4: To define the values of $\sigma$ on the remaining arrows $\iota_j$ (corresponding to ${\rad}P_{j}\rightarrow P_j$) and $\kappa_j$ (corresponding to $P_{j}\rightarrow P_j/{\soc}P_j$) for $1\leq j \leq m$ such that $\sigma(m_x)=0$ in $k\Gamma_{\Lambda}/J$ for all vertices $x$ in $\prescript{}{s}{\Gamma}_{\Lambda}$.}

Since $q(\sigma(\al^{0}_{3m-2}\be^{1}_{3m-2}+p))=\phi'_{0}(q(\al^{0}_{3m-2}\be^{1}_{3m-2}+p))=0,$ we have  $$qU(\sigma(\al^{0}_{3m-2}\be^{1}_{3m-2}+p))=Vq(\sigma(\al^{0}_{3m-2}\be^{1}_{3m-2}+p))=0$$
 and $U(\sigma(\al^{0}_{3m-2}\be^{1}_{3m-2}+p)):{\rad}P_1\rightarrow P_1/{\soc}P_1$ factors through a projective module. Since $P_1$ is the projective cover of $P_1/{\soc}P_1$,  $U(\sigma(\al^{0}_{3m-2}\be^{1}_{3m-2}+p)):{\rad}P_1\rightarrow P_1/{\soc}P_1$ factors through $P_1$. Let $U(\sigma(\al^{0}_{3m-2}\be^{1}_{3m-2}+p))+vu=0$, where $u:{\rad}P_1\rightarrow P_1$, $v:P_1\rightarrow P_1/{\soc}P_1$. Define $\sigma(\iota_1)$, $\sigma(\kappa_1)$ be morphisms in $k\Gamma_{\Lambda}/J$ such that $U(\sigma(\iota_1))=u$, $U(\sigma(\kappa_1))=v$. Then $U(\sigma(m_{(0,3m-1)}+p))=0$ and $\sigma(m_{(0,3m-1)}+p)=0$. Similarly one can define $\sigma(\iota_i)$, $\sigma(\kappa_i)$ for all $2\leq i\leq m$ such that $\sigma(m_x)=0$ in $k\Gamma_{\Lambda}/J$ for each vertex $x$ which corresponds to the radical of some $P_{i}$ for $2\leq i\leq m$.

As a summary, since $q\sigma=\phi'_{0}q$ on all arrows of $\Gamma_{\Lambda}$ and $\sigma(m_x)=0$ in $k\Gamma_{\Lambda}/J$ for all vertex $x$ in $\prescript{}{s}{\Gamma}_{\Lambda}$, we have lifted the functor $\phi'_{0}:k\prescript{}{s}{\Gamma}_{\Lambda}/I\rightarrow k\prescript{}{s}{\Gamma}_{\Lambda}/I$ to a functor $\psi'_{0}:k\Gamma_{\Lambda}/J\rightarrow k\Gamma_{\Lambda}/J$ such that $\psi'_{0}(x)=x$ for each vertex $x$ in $k\Gamma_{\Lambda}$ and $\psi'_{0}(\al)=\sigma(\al)$ for each arrow $\al$ of $\Gamma_{\Lambda}$, which again induces a functor $\psi':\mathrm{ind}\Lambda\rightarrow \mathrm{ind}\Lambda$ making all the faces of the diagram

$$\xymatrix@R=1.3pc{
	k\Gamma_{\Lambda}/J \ar[rd]^{U}_{\sim} \ar[dd]_{q}\ar[rr]^{\psi'_{0}}& & k\Gamma_{\Lambda}/J\ar[rd]^{U}_{\sim}\ar@{.>}[dd]_(0.3){q} & \\
	& \mathrm{ind}\Lambda\ar[dd]_(0.3){q}\ar[rr]_(0.7){\psi'} & & \mathrm{ind}\Lambda \ar[dd]_{q}\\
	k\prescript{}{s}{\Gamma}_{\Lambda}/I \ar[rd]^{V}_{\sim}\ar@{.>}[rr]_(0.7){\phi'_{0}} & & k\prescript{}{s}{\Gamma}_{\Lambda}/I \ar@{.>}[rd]^{V}_{\sim}& \\
	& \underline{\mathrm{ind}}\Lambda\ar[rr]_{\phi'} & & \underline{\mathrm{ind}}\Lambda \\
}$$
$$\mbox{Figure }4$$
commutative.

\medskip
{\it Step 2: To show that $\psi':\mathrm{ind}\Lambda\rightarrow \mathrm{ind}\Lambda$ is an isomorphism.}

Let $\psi: \mathrm{mod}\Lambda\rightarrow \mathrm{mod}\Lambda$ be a functor induced by $\psi':\mathrm{ind}\Lambda\rightarrow \mathrm{ind}\Lambda$. To show that $\psi'$ is an isomorphism, it suffices to show that $\psi$ is an equivalence. Since $\psi(X)\cong X$ for all $\Lambda$-module $X$, $\psi$ satisfies condition (3) and the former half of the condition (2) of Lemma \ref{Lem, equivalence}. By Lemma \ref{Lem, equivalence}, it suffices to show that $\psi({\rad}_{\Lambda})\subseteq {\rad}_{\Lambda}$ and $\psi$ preserves irreducible morphisms between indecomposable $\Lambda$-modules. Since the radical of $k\Gamma_{\Lambda}/J$ is generated by arrows of $\Gamma_{\Lambda}$ and $\psi'_{0}(x)=x$ for each vertex $x$ of $\Gamma_{\Lambda}$, $\psi'_{0}$ sends the morphisms in the radical of $k\Gamma_{\Lambda}/J$ to the morphisms in the radical of $k\Gamma_{\Lambda}/J$.  Hence $\psi({\rad}_{\Lambda})\subseteq {\rad}_{\Lambda}$.

Next we show that $\psi$ preserves irreducible morphisms between indecomposable $\Lambda$-modules. Since the diagram

$$\xymatrix{
\mathrm{mod}\Lambda\ar[r]^{\psi}\ar[d]_{q} & \mathrm{mod}\Lambda\ar[d]^{q} \\
\underline{\mathrm{mod}}\Lambda\ar[r]^{\phi} & \underline{\mathrm{mod}}\Lambda \\
}$$
commutes up to natural isomorphisms and $\phi$ is an equivalence, $\psi$ maps irreducible morphism between indecomposable nonprojective modules to irreducible morphism between indecomposable nonprojective modules (cf. Lemma \ref{Lem, lift an irreducible morphism}). Then it suffices to show that $\psi$ preserves irreducible morphisms ${\rad}P\rightarrow P$ and $P\rightarrow P/{\soc}P$, where $P$ is an indecomposable projective module. Let $P$ be an indecomposable projective module and \begin{equation*} 0\rightarrow {\rad}P\xrightarrow[]{(f_1,f_2)^T} P\oplus {\rad}P/{\soc}P\xrightarrow[]{(g_1,g_2)} P/{\soc}P\rightarrow 0 \end{equation*} be an almost split sequence. Since $\psi$ preserves irreducible morphisms between indecomposable nonprojective modules, $\psi(f_2)$, $\psi(g_2)$ are irreducible. Then $\psi(f_2)$ is an epimorphism and $\psi(g_2)$ is a monomorphism. Since ${\rad}P/{\soc}P$ is the unique maximal submodule of $P/{\soc}P$, $im(\psi(g_2))\subseteq {\rad}P/{\soc}P$. Since $l_{\Lambda}(im(\psi(g_2)))=l_{\Lambda}({\rad}P/{\soc}P)$, $im(\psi(g_2))= {\rad}P/{\soc}P$. Since $\psi(f_2)$ is an epimorphism, $im(\psi(g_{2}f_{2}))={\rad}P/{\soc}P$. Since $\psi(g_{1}f_{1})=-\psi(g_{2}f_{2})$, we have $im(\psi(g_{1})\psi(f_{1}))= {\rad}P/{\soc}P$. By Lemma \ref{Lem, irreducible morphism}, $\psi(f_{1})$, $\psi(g_{1})$ are irreducible morphisms. Therefore $\psi$ preserves irreducible morphisms between indecomposable $\Lambda$-modules. By Lemma \ref{Lem, equivalence}, $\psi$ is an equivalence.
\end{proof}

\begin{Rem}\label{standard-counterpart} Notice that since the modified mesh relation $m_x$ for $x=(0,3m-1)$ starts at a point which belongs to the configuration $\mathcal{C}$, it allows us to adjust the modified mesh relation using the projective module (see Step 1.4 above), which is the same as we do for the usual mesh relations. Therefore, the method in Proposition \ref{Prop, lift} applies for the standard RFS algebra which has the isomorphic AR-quiver as $\Lambda$. A similar method applies for each type of standard RFS algebras, see Appendix A for a detailed explanation.
\end{Rem}

\section{Proof of the main result}

Combining Proposition \ref{Prop, lift} and some ideas from \cite{Asashiba2003,Dugas2015,CKL} we prove our main result Theorem \ref{main-result} in this section.

\begin{Def}\label{Def,stable-picard} {\rm(\cite[Definition 1]{Asashiba2003})}
Let $A$ be an algebra. The stable Picard group $\mathrm{StPic}(A)$ of $A$ is the group formed by natural isomorphism classes of stable auto-equivalences of $A$. Let $\mathrm{Pic}'(A)$ be the image of canonical homomorphism $\mathrm{Pic}(A)\rightarrow \mathrm{StPic}(A)$, where $\mathrm{Pic}(A)$ denotes the Picard group formed by natural isomorphism classes of Morita auto-equivalences of $A$.
\end{Def}

The following result was proved in \cite{CKL} using the mutation theory of simple-minded systems. For the completeness we give its proof in Appendix B.

\begin{Prop}\label{H-and-H'}  {\rm(\cite[Lemma 4.10]{CKL})}
	Let $k$ be an algebraically closed field of characteristic $2$, $\Lambda$ be the representative algebra of nonstandard RFS algebras of type $(D_{3m},1/3,1)$, where $m\geq 2$. Then there exists a standard derived auto-equivalence of $\Lambda$ which induces a stable auto-equivalence $H$ of $\Lambda$ such that $H$ induces the automorphism of $\prescript{}{s}{\Gamma}_{\Lambda}$ defined by the swap of the two high vertices.
\end{Prop}
	
\xymatrix@R=1.0pc@C=0.9pc@W=0mm{ &&&&&&&&&\bullet\ar[dr]&&\bullet\ar[dr]\ar[dr]^{\gamma}&&\bullet\ar[dr]&&\bullet\ar[dr]& \\
 &&&&&&&&\bullet\ar[ur]\ar[dr]\ar[r]\ar@{{}*{\cdot}{}}[l]&\bullet\ar[r]&\bullet\ar[ur]^{\gamma'}\ar[dr]|{\delta_{3m-3}}\ar[r]^{\al'}&\bullet\ar[r]^{\al}&\bullet\ar[ur]^{\kappa}\ar[dr]\ar[r]^{\be}&\bullet\ar[r]^{\be'}&\bullet\ar[ur]\ar[dr]\ar[r]&\bullet\ar[r]&\bullet\ar@{{}*{\cdot}{}}[r]& \\
	&&&&&&&&&\bullet\ar[ur]|{\xi_{3m-3}\quad}\ar[dr]|{\delta_{3m-4}}&&\bullet\ar[ur]|{\quad\eta_{3m-3}}&&\bullet\ar[ur]&&\bullet\ar[ur]& \\
	&&&&&&&&\bullet\ar[ur]|{\xi_{3m-4}\quad}\ar[dr]|{\delta_{3m-5}\quad}&&\bullet\ar[ur]|{\quad\eta_{3m-4}}&&&&&& \\
	&&&&&&&&&\bullet\ar[ur]|{\quad\eta_{3m-5}}&&&&&&& \\
	&&&&&&&&&&&&&&&& \\
	&&&&&\bullet\ar[dr]|{\delta_3}\ar@{.}[uuurrr]&&&&&&&&&&& \\
	&&&&\bullet\ar[ur]|{\xi_{3}}\ar[dr]|{\delta_2}&&\bullet\ar@{.}[uuurrr]&&&&&&&&&& \\
	&&&\bullet\ar[ur]|{\xi_{2}}\ar[dr]|{\delta_1}&&\bullet\ar[ur]|{\eta_{2}}&&&&&&&&&&& \\
	\ar@{.}[uuuuuuuurrrrrrrr]&&\bullet\ar[ur]|{\xi_{1}}&&\bullet\ar[ur]|{\eta_{1}}&&\ar@{.}[uuuuuuurrrrrrr]&&&&&&&&&& \\	
	-1&&0&&1&&2&&&&&&&&&&
}
$$\mbox{Figure }5$$

\begin{Rem}
Using the fact that $\underline{\mathrm{ind}}\Lambda$ is equivalent to $k\prescript{}{s}{\Gamma}_{\Lambda}/I$, one can give a concrete construction of the stable auto-equivalence $H$ up to a Morita equivalence in viewing of Proposition \ref{Prop, lift}. For each $2\leq i\leq 3m-2$, let $q_{i-1}$ be the path $(0,i)\rightarrow (1,i-1)\rightarrow (1,i)\rightarrow (2,i-1)\rightarrow(2,i)\rightarrow\dots\rightarrow (2m-1,i-1)\rightarrow (2m-1,i)$. For each $0\leq i\leq 2m-2$, let $l_i$ be the path $(i,3m-2)\rightarrow (i,3m)\rightarrow (i+1,3m-2)$, $h_i$ be the path $(i,3m-2)\rightarrow (i,3m-1)\rightarrow (i+1,3m-2)$, $p_i$ be the path $(i,3m-2)\rightarrow (i+1,3m-3)\rightarrow (i+1,3m-2)$. Define a functor $H':k\prescript{}{s}{\Gamma}_{\Lambda}\rightarrow k\prescript{}{s}{\Gamma}_{\Lambda}$ by $H'(x)=\eta(x)$ for each vertex $x$ of $\prescript{}{s}{\Gamma}_{\Lambda}$, where $\eta$ is the automorphism on $\prescript{}{s}{\Gamma}_{\Lambda}$ by the swap of the two high vertices (see Figure 5), and the definition of $H'$ on arrows are given as follows: $H'(\al):=\gamma+l_{2m-1}h_{2m-2}\dots l_{3}h_{2}l_{1}\gamma$, $H'(\gamma):=\al+h_{2m-1}l_{2m-2}\dots h_{3}l_{2}h_{1}\al$, $H'(\delta_{i}):=\delta_{i}+\delta_{i}q_i$, where $1\leq i\leq 3m-3$, $H'(\zeta):=\eta(\zeta)$ for other arrows $\zeta$ in $\prescript{}{s}{\Gamma}_{\Lambda}$. Then it is straightforward to verify that $H'$ preserves all the modified mesh relations and therefore induces a functor $H: k\prescript{}{s}{\Gamma}_{\Lambda}/I\rightarrow k\prescript{}{s}{\Gamma}_{\Lambda}/I$. Moreover, $H$ preserves the radical ${\rad}(~,~)$ and irreducible morphisms, it follows that $H$ is an equivalence (cf. Lemma \ref{Lem, equivalence}).
\end{Rem}

The following result should be compared with \cite[Theorem 3.1]{Asashiba2003} (or Corollary \ref{Prop,stable-picard,standard}) for standard RFS algebras.

\begin{Prop}\label{Prop,stable-picard}
Let $k$ be an algebraically closed field of characteristic $2$, $\Lambda$ be the representative algebra of nonstandard RFS algebras of type $(D_{3m},1/3,1)$, where $m\geq 2$. For any stable auto-equivalence $\phi$ of $\Lambda$, we denote by $[\phi]$ its natural isomorphism class. Then $\mathrm{StPic}(\Lambda)=(\mathrm{Pic}'(\Lambda)\cdot\langle [\Omega_{\Lambda}]\rangle)\cup (\mathrm{Pic}'(\Lambda)\cdot\langle [\Omega_{\Lambda}]\rangle)[H]$, where $\Omega_{\Lambda}$ is the loop functor, and $H$ is a stable auto-equivalence of $\Lambda$ as defined in Proposition \ref{H-and-H'}, which satisfies $[H]^2\in \mathrm{Pic}'(\Lambda)$, and $\langle [\Omega_{\Lambda}]\rangle$ denotes the cyclic subgroup of $\mathrm{StPic}(\Lambda)$ generated by $[\Omega_{\Lambda}]$.
\end{Prop}

\begin{proof} Note that each stable auto-equivalence $\phi$ of $\Lambda$ induces an automorphism $f$ of $\prescript{}{s}{\Gamma}_{\Lambda}$ as a translation quiver (cf. \cite[Chapter X, Corollary 1.9]{ARS}). According to \cite{Riedtmann1} (see also \cite[Proposition 2.1]{Asashiba2003}), $Aut(\mathbb{Z}D_{3m})=\langle \tau\rangle\times\langle\eta'\rangle$, where $\eta'$ is the automorphism of $\mathbb{Z}D_{3m}$ which is induced from the automorphism of the quiver $D_{3m}$ defined by the swap of the two high vertices. Then $Aut(\prescript{}{s}{\Gamma}_{\Lambda})=\langle \tau\rangle\times\langle\eta\rangle$, where $\eta$ is the automorphism of $\prescript{}{s}{\Gamma}_{\Lambda}$ induced from $\eta'$.

Let $f=\tau^{r}\eta^{i}\in Aut(\prescript{}{s}{\Gamma}_{\Lambda})$ (where $i=0$ or $1$) be induced from a stable auto-equivalence $\phi$ of $\Lambda$. Then the automorphism of $\prescript{}{s}{\Gamma}_{\Lambda}$ induced by  the stable auto-equivalence $\tau^{-r}_{\Lambda}\phi H^{-i}$ acts as the identity map of the set of vertices, where $H^{-}$ is a quasi-inverse of $H$. By Proposition \ref{Prop, lift}, $\tau^{-r}_{\Lambda}\phi H^{-i}$ lifts to a Morita equivalence. Then $[\tau^{-r}_{\Lambda}\phi H^{-i}]\in \mathrm{Pic}'(\Lambda)$. Since $[\tau_{\Lambda}]=[\Omega_{\Lambda}^{2}]$, $[\tau_{\Lambda}]\in \langle [\Omega_{\Lambda}]\rangle$. Therefore $[\phi]=[\tau_{\Lambda}]^{r}[\tau^{-r}_{\Lambda}\phi H^{-i}][H]^{i}\in (\mathrm{Pic}'(\Lambda)\cdot\langle [\Omega_{\Lambda}]\rangle)\cup (\mathrm{Pic}'(\Lambda)\cdot\langle [\Omega_{\Lambda}]\rangle)[H]$. The fact that $[H]^2\in \mathrm{Pic}'(\Lambda)$ also follows from Proposition \ref{Prop, lift}.
\end{proof}

\begin{Rem}\label{coset-decomposition}
Sometimes the two cosets $\mathrm{Pic}'(\Lambda)\cdot\langle [\Omega_{\Lambda}]\rangle$ and $(\mathrm{Pic}'(\Lambda)\cdot\langle [\Omega_{\Lambda}]\rangle)[H]$ of the subgroup $\mathrm{Pic}'(\Lambda)\cdot\langle [\Omega_{\Lambda}]\rangle$ of the stable Picard group $\mathrm{StPic}(\Lambda)$ are the same. For example, in the case $m=3$, $[\Omega_{\Lambda}]\in \mathrm{Pic}'(\Lambda)\cdot[\tau_{\Lambda}]^{3}[H]$ and therefore $[H]\in \mathrm{Pic}'(\Lambda)\cdot\langle [\Omega_{\Lambda}]\rangle$.
\end{Rem}

\begin{Prop}\label{lift to s.d.e}
Let $k$ be an algebraically closed field of characteristic $2$, $\Lambda$ be the representative algebra of nonstandard RFS algebras of type $(D_{3m},1/3,1)$, where $m\geq 2$. Then each stable auto-equivalence of $\Lambda$ lifts to a standard derived equivalence.
\end{Prop}

\begin{proof} It follows from Proposition \ref{H-and-H'} and Proposition \ref{Prop,stable-picard}.
\end{proof}

\medskip
\begin{proof}[{\bf Proof of Theorem \ref{main-result}}] By Theorem \ref{dclassRFS}, both $A$ and $B$ are derived equivalent to the same nonstandard RFS representative algebra $\Lambda$. Then there exists stable equivalences $\xi:\underline{\mathrm{mod}}A\rightarrow \underline{\mathrm{mod}}\Lambda$ and $\eta:\underline{\mathrm{mod}}B\rightarrow \underline{\mathrm{mod}}\Lambda$ such that $\xi$, $\eta$ lift to standard derived equivalences. By Proposition \ref{lift to s.d.e}, $\eta\phi{\xi}^{-1}:\underline{\mathrm{mod}}\Lambda\rightarrow \underline{\mathrm{mod}}\Lambda$ lifts to a standard derived equivalence. Then $\phi={\eta}^{-1}(\eta\phi{\xi}^{-1})\xi$ lifts to a standard derived equivalence.
\end{proof}

\newpage
\appendix
\section{}

Throughout this appendix we fix the enumeration on the vertices of $A_{n}$, $D_{n}$, $E_{n}$ as follow:

$$\xymatrix{A_{n}: & & &
               1 \ar[r] & 2\ar[r]  & \cdots  \ar[r] & n-1 \ar[r] &n
	}$$

$$\xymatrix{D_{n}:\quad & &
		& & &n & \\
		& & 1 \ar[r] & 2\ar[r]  & \cdots  \ar[r] & n-2 \ar[u] \ar[r] &n-1
	}$$

$$\xymatrix{E_{n}: &
		& & &n & & \\
		& 1 \ar[r] & 2\ar[r]  & \cdots \ar[r] & n-3 \ar[u]\ar[r] & n-2 \ar[r] &n-1
	}$$

The main purpose of this appendix is to prove the following result, which is a corrected form of \cite[Proposition 3.3]{Asashiba2003}. We are grateful to the referee who suggests to add this content.

\begin{Prop}\label{corrected-form}  Let $k$ be an algebraically closed field, and let $A$ be some properly selected representative algebra of standard RFS algebras. Let $\phi:\underline{\mathrm{mod}}A\rightarrow\underline{\mathrm{mod}}A$ be a stable equivalence such that $\phi (X)\cong X$ for any $X\in\underline{\mathrm{mod}}A$. Then $\phi$ lifts to a Morita equivalence.
\end{Prop}

Using Proposition \ref{corrected-form}, we can reprove \cite[Theorem 3.1]{Asashiba2003}, whose original proof uses \cite[Proposition 3.3]{Asashiba2003}. Note that the main result \cite[Main Theorem]{Asashiba2003} follows from \cite[Theorem 3.1]{Asashiba2003}. By the same reason as in Remark \ref{counterexample}(1), here we also need to assume that the considered algebra has Loewy length greater than $2$.

\begin{Cor}\label{Prop,stable-picard,standard} {\rm(\cite[Theorem 3.1]{Asashiba2003})}
Let $A$ be the representative algebra of representation-finite standard RFS algebras in Proposition \ref{corrected-form} with Loewy length greater than $2$. If $A$ is not of type $(D_{3m},s/3,1)$ with $m\geq 2$ and $3\nmid s\geq 1$, then \begin{equation*}\mathrm{StPic}(A)=\mathrm{Pic}'(A)\cdot\langle [\Omega_{A}]\rangle.\end{equation*} \\ If $A$ is of type $(D_{3m},s/3,1)$ with $m\geq 2$ and $3\nmid s\geq 1$, then \begin{equation*}\mathrm{StPic}(A)=(\mathrm{Pic}'(A)\cdot\langle [\Omega_{A}]\rangle)\cup (\mathrm{Pic}'(A)\cdot\langle [\Omega_{A}]\rangle)[H],\end{equation*} where $H$ is a stable auto-equivalence of $A$ induced from the automorphism of $\prescript{}{s}{\Gamma}_{\Lambda}$ defined by the swap of the two high vertices, which satisfies $[H]^2\in \mathrm{Pic}'(A)$.
\end{Cor}

\begin{proof} Note that the representative algebra $A$ has Loewy length $2$ if and only if $A$ has type $\mathrm{typ}(A)=(A_{1},s,1)$ with $s\geq 1$, so we exclude this type in the following proof.

For a given $A$, let $\mathcal{C}$ be the set of vertices in the stable AR-quiver $\prescript{}{s}{\Gamma}_{A}$ which correspond to radicals of indecomposable projective $A$-modules. According to the proof of Proposition \ref{corrected-form}, we list the positions of $\mathcal{C}$ in $\prescript{}{s}{\Gamma}_{A}$ for each representative algebra $A$ as follows.

\begin{itemize}

\item If $typ(A)=(A_{n},s/n,1)$ with $s,n\geq 1$, then $\mathcal{C}=\{(i,n)\mid 0\leq i\leq s-1\}$.

\item If $typ(A)=(A_{2p+1},s,2)$ with $s,p\geq 1$, then $\mathcal{C}=\{((2p+1)i+j,1),((2p+1)i+p,p+1),$ \\ $((2p+1)i+j+p+1,2p+1)\mid 0\leq i\leq s-1, 0\leq j\leq p-1\}$.

\item If $typ(A)=(D_{n},s,1)$ or $typ(A)=(D_{n},s,2)$ with $n\geq 4$ and $s\geq 1$, then $\mathcal{C}=\{((2n-3)i,n-1),((2n-3)i,n),((2n-3)i+n-1,n-2),((2n-3)i+j,1),\mid 0\leq i\leq s-1, 1\leq j\leq n-3\}$.

\item If $typ(A)=(D_{4},s,3)$ with $s\geq 1$, then $\mathcal{C}=\{(5i,3),(5i,4),(5i+3,2),(5i+1,1)\mid 0\leq i\leq s-1\}$.

\item If $typ(A)=(D_{3m},s/3,1)$ with $m\geq 2$ and $3\nmid s\geq 1$, then $\mathcal{C}=\{((2m-1)i,3m-1), ((2m-1)i+j,1)\mid 0\leq i\leq s-1, m\leq j\leq 2m-2\}$.

\item If $typ(A)=(E_{n},s,1)$ or $typ(A)=(E_{6},s,2)$ with $6\leq n \leq 8$ and $s\geq 1$, then $\mathcal{C}=\{(m_{n}\cdot i+j,1),(m_{n}\cdot i-1,n),(m_{n}\cdot i-2,n-1),(m_{n}\cdot i-1,n-1),(m_{n}\cdot i+(m_{n}-1)/2,n-3)\mid 0\leq i\leq s-1, 0\leq j\leq n-5\}$, where $m_{6}=11$, $m_{7}=17$, $m_{8}=29$.

\end{itemize}

Suppose that $\mathrm{typ}(A)\notin\{(A_{n},s/n,1),(D_{3m},r/3,1)\mid n,r,s\geq 1,m\geq 2, 3\nmid r\}$. Combining a result in \cite{Riedtmann1} (see also \cite[Proposition 2.1]{Asashiba2003}), we can directly prove that each automorphism of $\prescript{}{s}{\Gamma}_{A}$ (as a translation quiver) is of the form $\tau^{a}\rho$, where $\rho$ is an automorphism of $\prescript{}{s}{\Gamma}_{A}$ such that $\mathcal{C}$ is stable under $\rho$. Let $\phi:\underline{\mathrm{mod}}A\rightarrow\underline{\mathrm{mod}}A$ be a stable equivalence which induces an automorphism $f$ of $\prescript{}{s}{\Gamma}_{A}$ (as a translation quiver). Assume $f=\tau^{a}\rho$ with $\mathcal{C}$ stable under $\rho$, then $\rho$ extends to an automorphism of $\Gamma_{A}$, which induces an auto-equivalence of $k(\Gamma_{A})$. Thus, there exists a Morita equivalence $\Psi:\mathrm{mod}A\rightarrow \mathrm{mod}A$ which induces a stable equivalence $\psi:\underline{\mathrm{mod}}A\rightarrow\underline{\mathrm{mod}}A$, such that the automorphism of $\prescript{}{s}{\Gamma}_{A}$ induced by $\psi$ is $\rho$. Since $\phi(\tau_{A}^{a}\psi)^{-1}$ induces identity automorphism of $\prescript{}{s}{\Gamma}_{A}$, by Proposition \ref{corrected-form}, it lifts to a Morita equivalence. Since $[\phi]=[\phi(\tau_{A}^{a}\psi)^{-1}][\tau_{A}]^{a}[\psi]$ with $[\phi(\tau_{A}^{a}\psi)^{-1}],[\psi]\in \mathrm{Pic}'(A)$ and $[\tau_{A}]\in\mathrm{Pic}'(A)\cdot\langle [\Omega_{A}]\rangle$, $[\phi]\in\mathrm{Pic}'(A)\cdot\langle [\Omega_{A}]\rangle$.

Suppose that $\mathrm{typ}(A)=(A_{n},s/n,1)$ with $n> 1, s\geq 1$. Then each automorphism of $\prescript{}{s}{\Gamma}_{A}$ (as a translation quiver) is of the form $\tau^{a}$ or $\tau^{a}\rho$, where $\rho$ is given by $(p,q)\mapsto (p+q-1,n+1-q)$. It can be shown that the automorphism of $\prescript{}{s}{\Gamma}_{A}$ induced by $\Omega_{A}$ is $\tau^{b}\rho$ for some $b$. Let $\phi:\underline{\mathrm{mod}}A\rightarrow\underline{\mathrm{mod}}A$ be a stable equivalence which induces an automorphism $f$ of $\prescript{}{s}{\Gamma}_{A}$ (as a translation quiver). If $f=\tau^{a}$ for some $a$, by Proposition \ref{corrected-form}, $\phi\tau_{A}^{-a}$ lifts to a Morita equivalence. Then $[\phi]=[\phi\tau_{A}^{-a}][\tau_{A}]^{a}\in \mathrm{Pic}'(A)\cdot\langle [\Omega_{A}]\rangle$. If $f=\tau^{a}\rho$ for some $a$, by Proposition \ref{corrected-form}, $\phi\Omega_{A}^{-1}\tau_{A}^{b-a}$ lifts to a Morita equivalence. Then $[\phi]=[\phi\Omega_{A}^{-1}\tau_{A}^{b-a}][\tau_{A}]^{a-b}[\Omega_{A}]\in \mathrm{Pic}'(A)\cdot\langle [\Omega_{A}]\rangle$.

Suppose that $\mathrm{typ}(A)=(D_{3m},s/3,1)$ with $m\geq 2$ and $3\nmid s\geq 1$. Then each automorphism of $\prescript{}{s}{\Gamma}_{A}$ (as a translation quiver) is of the form $\tau^{a}$ or $\tau^{a}\eta$, where $\eta$ is the automorphism of $\prescript{}{s}{\Gamma}_{\Lambda}$ defined by the swap of the two high vertices. By the same method, it can be shown that for each stable auto-equivalence $\phi$ of $A$, $[\phi]\in\mathrm{Pic}'(A)\cdot\langle [\Omega_{A}]\rangle$ or $[\phi]\in(\mathrm{Pic}'(A)\cdot\langle [\Omega_{A}]\rangle)[H]$. The fact that $[H]^2\in \mathrm{Pic}'(A)$ also follows from Proposition \ref{corrected-form}.
\end{proof}

We now turn to the proof of Proposition \ref{corrected-form}. For each type $(Q,f,t)$ of standard RFS algebras, Asashiba gave a representative algebra $\Lambda(Q,f,t)$ inside its derived equivalence class, all the representative algebras are listed in \cite[Appendix 2]{Asashiba2003}. Unless otherwise stated, we will choose the representative algebra $A$ in Proposition \ref{corrected-form} as $\Lambda(Q,f,t)$.

Since $A$ is standard, there is a well-behaved isomorphism $U:k(\Gamma_{A})\rightarrow \mathrm{ind}A$ such that it maps each vertex of $\Gamma_{A}$ to the corresponding indecomposable module and maps each arrow of $\Gamma_{A}$ to an irreducible morphism; moreover, $U$ induces a well-behaved isomorphism $V: k(\prescript{}{s}{\Gamma}_{A})\simeq \underline{\mathrm{ind}}A$. Therefore, we can adopt the method in Proof of Proposition \ref{Prop, lift} to give a proof of Proposition \ref{corrected-form}. By analysing the proof of Proposition \ref{Prop, lift}, we know that if we can construct a functor $\Phi: k(\Gamma_{A})\rightarrow k(\Gamma_{A})$ from a given isomorphism functor $\phi': k(\prescript{}{s}{\Gamma}_{A})\rightarrow k(\prescript{}{s}{\Gamma}_{A})$  with $\phi'(x)=x$ for all $x\in \prescript{}{s}{\Gamma}_{A}$, then $\Phi$ becomes automatically an isomorphism functor under our assumption. Thus we can reduce the proof of Proposition \ref{corrected-form} to the construction of a functor $\Phi: k(\Gamma_{A})\rightarrow k(\Gamma_{A})$ lifting $\phi'$.

We shall give the construction in each type using the similar idea as Step 1 in the proof of Proposition \ref{Prop, lift}. Recall that Step 1 in the proof of Proposition \ref{Prop, lift} divides into four substeps (from Step 1.1 to Step 1.4), however, we do not need Step 1.3 in most cases except for the type $(D_{3m},s/3,1)$ with $m\geq 2$ and $3\nmid s\geq 1$.

\medskip
In the following, $\mathcal{C}$ is always assumed to be the set of vertices in $\prescript{}{s}{\Gamma}_{A}$ which correspond to radicals of indecomposable projective $A$-modules and $\star$ denotes the positions of $\mathcal{C}$.

\medskip
{\bf 1. Type $(A_{n},s/n,1)$ with $s,n\geq 1$.} Let $A=\Lambda (A_{n},s/n,1)$ be the self-injective Nakayama algebra given by the quiver below
 with relations $\alpha_{i+n}\cdots\alpha_{i+1}\alpha_{i}=0$ for all $i\in \{1,2,\cdots,s\}=\mathbb{Z}/\langle s \rangle$.

$$
\vcenter{
	\xymatrix@R=1.9pc@C=1.9pc{
	& s \ar[dl]_{\alpha_s} & \ar[l]_{\alpha_{s-1}} \ar@{{}*{\cdot}{}}[r] &  \\
	1  \ar[dr]_{\alpha_1} & & & \\
	& 2 \ar[r]_{\alpha_2} & \ar@{{}*{\cdot}{}}[r]&\\
}}$$

Then $\prescript{}{s}{\Gamma}_{A}\cong\mathbb{Z}A_{n}/\langle\tau^{s}\rangle$ is of the form:

$$
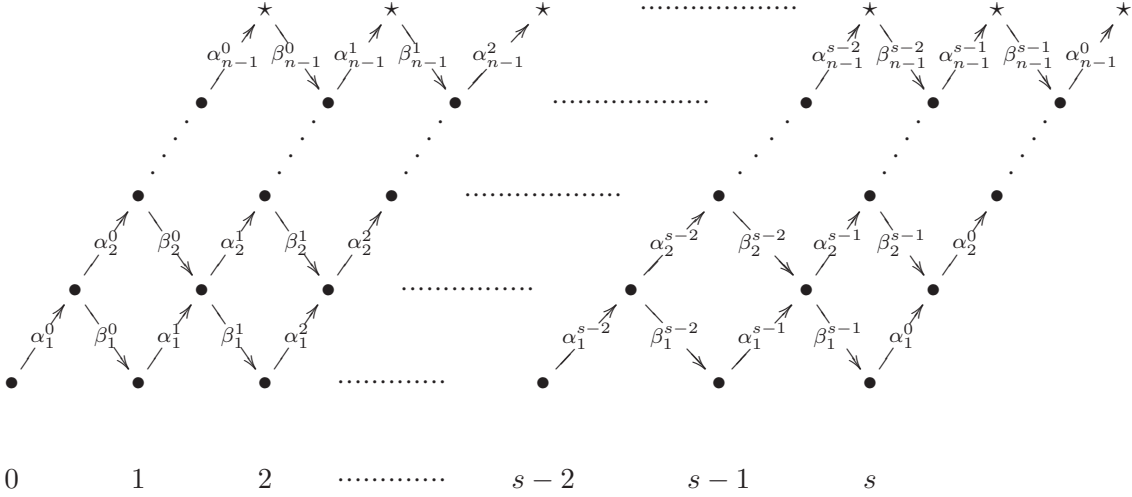
\vcenter{
	\xymatrix@R=2.0pc@C=1.0pc {
	&&&&\star\ar[dr]|{\beta^{0}_{n-1}} &&\star\ar[dr]|{\beta^{1}_{n-1}} &&\star &\ar@{{}*{\cdot}{}}[rr]&&&\star\ar[dr]|{\beta^{s-2}_{n-1}} &&\star\ar[dr]|{\beta^{s-1}_{n-1}} && \star  \\
	&&&\bullet\ar[ur]|{\alpha^{0}_{n-1}} &&\bullet\ar[ur]|{\alpha^{1}_{n-1}} &&\bullet\ar[ur]|{\alpha^{2}_{n-1}} &\ar@{{}*{\cdot}{}}[rr]&&&\bullet\ar[ur]|{\alpha^{s-2}_{n-1}} &&\bullet\ar[ur]|{\alpha^{s-1}_{n-1}} && \bullet\ar[ur]|{\alpha^{0}_{n-1}} &\\
	&&\bullet\ar@{{}*{\cdot}{}}[ur]\ar[dr]|{\beta^{0}_{2}} &&\bullet\ar@{{}*{\cdot}{}}[ur]\ar[dr]|{\beta^{1}_{2}} &&\bullet\ar@{{}*{\cdot}{}}[ur] &\ar@{{}*{\cdot}{}}[rr]&&&\bullet\ar@{{}*{\cdot}{}}[ur]\ar[dr]|{\beta^{s-2}_{2}} &&\bullet\ar@{{}*{\cdot}{}}[ur]\ar[dr]|{\beta^{s-1}_{2}} && \bullet\ar@{{}*{\cdot}{}}[ur] &&\\
    &\bullet\ar[ur]|{\alpha^{0}_{2}}\ar[dr]|{\beta^{0}_{1}} &&\bullet\ar[ur]|{\alpha^{1}_{2}}\ar[dr]|{\beta^{1}_{1}} &&\bullet\ar[ur]|{\alpha^{2}_{2}} &\ar@{{}*{\cdot}{}}[rr]&&&\bullet\ar[ur]|{\alpha^{s-2}_{2}}\ar[dr]|{\beta^{s-2}_{1}} &&\bullet\ar[ur]|{\alpha^{s-1}_{2}}\ar[dr]|{\beta^{s-1}_{1}} && \bullet\ar[ur]|{\alpha^{0}_{2}} &&&\\
    \bullet\ar[ur]|{\alpha^{0}_{1}} &&\bullet\ar[ur]|{\alpha^{1}_{1}} &&\bullet\ar[ur]|{\alpha^{2}_{1}} &\ar@{{}*{\cdot}{}}[rr]&&&\bullet\ar[ur]|{\alpha^{s-2}_{1}} &&\bullet\ar[ur]|{\alpha^{s-1}_{1}} && \bullet\ar[ur]|{\alpha^{0}_{1}} &&&&\\
    0 &&1 &&2 &\ar@{{}*{\cdot}{}}[rr]&&&s-2 &&s-1 && s &&&&
}}$$

By the position of $\mathcal{C}$ in $\prescript{}{s}{\Gamma}_{A}$, one can show that in $\prescript{}{s}{\Gamma}_{A}$ the upward arrows correspond to irreducible monomorphisms and the downward arrows correspond to irreducible epimorphisms. Choose a section $A_{n}'$ in $\prescript{}{s}{\Gamma}_{A}$ as follow:

$$
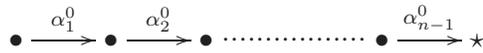
\vcenter{
\xymatrix {
\bullet\ar[r]^{\alpha^{0}_{1}}&\bullet\ar[r]^{\alpha^{0}_{2}}&\bullet\ar@{{}*{\cdot}{}}[rr]&&\bullet\ar[r]^{\alpha^{0}_{n-1}}&\star
}}
$$

Let $\phi': k(\prescript{}{s}{\Gamma}_{A})\rightarrow k(\prescript{}{s}{\Gamma}_{A})$ be an isomorphism which maps each object in $k(\prescript{}{s}{\Gamma}_{A})$ to itself.  To lift $\phi': k(\prescript{}{s}{\Gamma}_{A})\rightarrow k(\prescript{}{s}{\Gamma}_{A})$ to a functor $\Phi:k(\Gamma_{A})\rightarrow k(\Gamma_{A})$, one can first choose morphisms $\Phi(\alpha^{0}_{1}),\cdots,\Phi(\alpha^{0}_{n-1})$ which lift $\phi'(\alpha^{0}_{1}),\cdots,\phi'(\alpha^{0}_{n-1})$ respectively. Using Lemma \ref{Lem, lift a mesh relation}(2) (and similar result as Lemma \ref{Lem, preserve irreducible morphism} for $A$), one can lift arrows in $\prescript{}{s}{\Gamma}_{A}$ from the section $A_{n}'$ to the right. Now assume that the values of $\Phi$ on all arrows of $\prescript{}{s}{\Gamma}_{A}$ except $\beta^{s-1}_{1},\cdots,\beta^{s-1}_{n-1}$ have been defined, which satisfy $\Phi(m_x)=0$ for each vertex $x$ such that $x$ is not in $\mathcal{C}$ and such that the values of $\Phi$ on all arrows in $m_x$ have been defined (which corresponds to Step 1.1 in Proof of Proposition \ref{Prop, lift}). Since $\alpha^{s-1}_{1},\cdots,\alpha^{s-1}_{n-1}$ correspond to irreducible monomorphisms, by Lemma \ref{Lem, lift a mesh relations}(2), one can define $\Phi(\beta^{s-1}_{1}),\cdots,\Phi(\beta^{s-1}_{n-1})$ from the bottom to the top such that $\Phi(m_x)=0$ for each vertex $x$ which is not in $\mathcal{C}$ (which corresponds to Step 1.2 in Proof of Proposition \ref{Prop, lift}). Finally, we define the values of $\Phi$ on the arrows of ${\Gamma}_{A}$ which link to projective vertices (which corresponds to Step 1.4 in Proof of Proposition \ref{Prop, lift}).

\medskip
{\bf 2. Type $(A_{2p+1},s,2)$ with $s,p\geq 1$.} Let $A=\Lambda (A_{2p+1},s,2)$ be the canonical M\"{o}bius algebra given by the quiver below with relations

\begin{enumerate}
\item $\alpha^{i}_{p}\cdots\alpha^{i}_{0}=\beta^{i}_{p}\cdots\beta^{i}_{0}$ for all $i\in \{0,\cdots,s-1\}$;
\item $\beta^{i+1}_{0}\alpha^{i}_{p}=\alpha^{i+1}_{0}\beta^{i}_{p}=0$ for all $i\in \{0,\cdots,s-2\}$ and $\alpha^{0}_{0}\alpha^{s-1}_{p}=\beta^{0}_{0}\beta^{s-1}_{p}=0$;
\item Paths of length $p+2$ are equal to $0$.
\end{enumerate}

$$
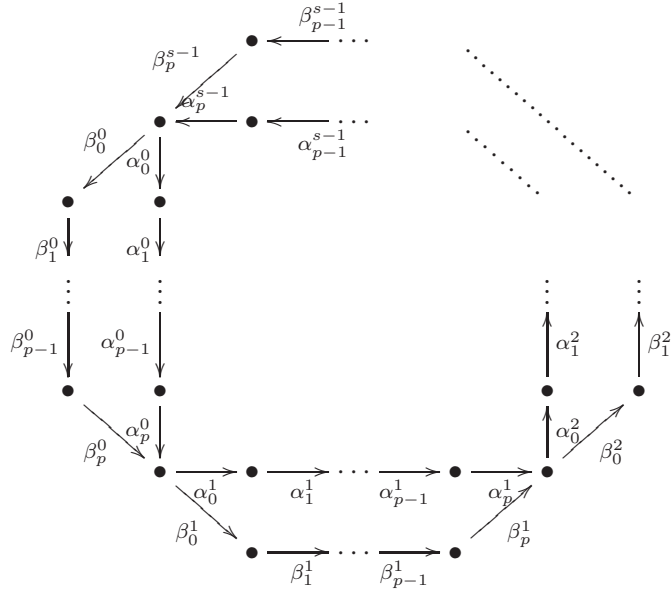
\vcenter{
\xymatrix@R=1.6pc@C=1.9pc{
&&\bullet\ar[dl]_{\beta^{s-1}_{p}}&\cdots\ar[l]_(0.3){\beta^{s-1}_{p-1}}&\ar@{{}*{\cdot}{}}[ddrr]&& \\
&\bullet\ar[dl]_{\beta^{0}_{0}}\ar[d]_{\alpha^{0}_{0}}&\bullet\ar[l]_{\alpha^{s-1}_{p}}&
\cdots\ar[l]^(0.3){\alpha^{s-1}_{p-1}}&\ar@{{}*{\cdot}{}}[dr]&& \\
\bullet\ar[d]_{\beta^{0}_{1}}&\bullet\ar[d]_{\alpha^{0}_{1}}&&&&& \\
\vdots\ar[d]_{\beta^{0}_{p-1}}&\vdots\ar[d]_{\alpha^{0}_{p-1}}&&&&\vdots&\vdots \\
\bullet\ar[dr]_{\beta^{0}_{p}}&\bullet\ar[d]_{\alpha^{0}_{p}}&&&&\bullet\ar[u]_{\alpha^{2}_{1}}&\bullet\ar[u]_{\beta^{2}_{1}} \\
&\bullet\ar[dr]_{\beta^{1}_{0}}\ar[r]_{\alpha^{1}_{0}}&\bullet\ar[r]_{\alpha^{1}_{1}}&\cdots\ar[r]_{\alpha^{1}_{p-1}}&\bullet\ar[r]_{\alpha^{1}_{p}}
&\bullet\ar[u]_{\alpha^{2}_{0}}\ar[ur]_{\beta^{2}_{0}}& \\
&&\bullet\ar[r]_{\beta^{1}_{1}}&\cdots\ar[r]_{\beta^{1}_{p-1}}&\bullet\ar[ur]_{\beta^{1}_{p}}&&
}}
$$

Let $\eta$ be the automorphism of $\mathbb{Z}A_{2p+1}$ given by $(m,n)\mapsto (m+n-1-p,2p+2-n)$. Then part of $\prescript{}{s}{\Gamma}_{A}\cong\mathbb{Z}A_{2p+1}/\langle\tau^{(2p+1)s}\eta\rangle$ is of the form:

$$
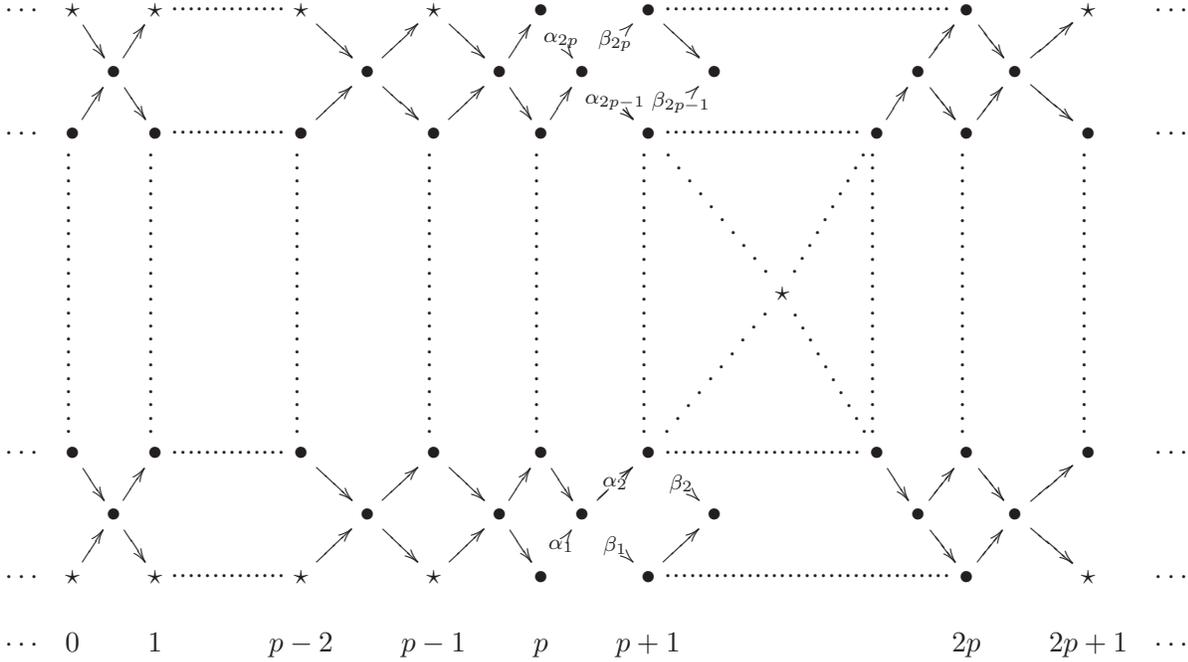
\vcenter{
\xymatrix@R=1pc@C=0.3pc {
\cdots&\star\ar[dr]&&\star\ar@{{}*{\cdot}{}}[rrrr]&&&&\star\ar[dr]&&\star\ar[dr]&&\bullet\ar[dr]|{\alpha_{2p}}
&&\bullet\ar[dr]\ar@{{}*{\cdot}{}}[rrrrrrrr]&&&&&&&&\bullet\ar[dr]&&\star&\cdots\\
&&\bullet\ar[dr]\ar[ur]&&&&&&\bullet\ar[dr]\ar[ur]&&\bullet\ar[dr]\ar[ur]&&\bullet\ar[dr]|{\alpha_{2p-1}}\ar[ur]|{\beta_{2p}}&&\bullet&&&&&&\bullet\ar[dr]
\ar[ur]&&\bullet\ar[dr]\ar[ur]&& \\
\cdots&\bullet\ar[ur]&&\bullet\ar@{{}*{\cdot}{}}[rrrr]&&&&\bullet\ar[ur]&&\bullet\ar[ur]&&\bullet\ar[ur]&&\bullet
\ar@{{}*{\cdot}{}}[dddrrr]\ar@{{}*{\cdot}{}}[rrrrrr]\ar[ur]|{\beta_{2p-1}}&&&&&&\bullet\ar[ur]&&\bullet\ar[ur]&&\bullet&\cdots \\
&&&&&&&&&&&&&&&&&&&&&& \\
&&&&&&&&&&&&&&&&&&&&&& \\
&&&&&&&&&&&&&&&&\star\ar@{{}*{\cdot}{}}[dddrrr]\ar@{{}*{\cdot}{}}[uuurrr]&&&&&&&& \\
&&&&&&&&&&&&&&&&&&&&&&&& \\
&&&&&&&&&&&&&&&&&&&&&&&& \\
\cdots&\bullet\ar@{{}*{\cdot}{}}[uuuuuu]\ar[dr]&&\bullet\ar@{{}*{\cdot}{}}[uuuuuu]\ar@{{}*{\cdot}{}}[rrrr]
&&&&\bullet\ar@{{}*{\cdot}{}}[uuuuuu]\ar[dr]&&\bullet\ar@{{}*{\cdot}{}}[uuuuuu]\ar[dr]&&
\bullet\ar[dr]\ar@{{}*{\cdot}{}}[uuuuuu]&&\bullet\ar@{{}*{\cdot}{}}[uuuuuu]
\ar@{{}*{\cdot}{}}[uuurrr]\ar@{{}*{\cdot}{}}[rrrrrr]\ar[dr]|{\beta_{2}}&&&&&&\bullet\ar@{{}*{\cdot}{}}[uuuuuu]\ar[dr]&&
\bullet\ar[dr]\ar@{{}*{\cdot}{}}[uuuuuu]&&\bullet\ar@{{}*{\cdot}{}}[uuuuuu]&\cdots \\
&&\bullet\ar[dr]\ar[ur]&&&&&&\bullet\ar[dr]\ar[ur]&&\bullet\ar[dr]\ar[ur]&&\bullet\ar[dr]|{\beta_{1}}\ar[ur]|{\alpha_2}&&\bullet&&&&&&\bullet\ar[dr]
\ar[ur]&&\bullet\ar[dr]\ar[ur]&&\\
\cdots&\star\ar[ur]&&\star\ar@{{}*{\cdot}{}}[rrrr]&&&&\star\ar[ur]&&\star\ar[ur]&&\bullet\ar[ur]|{\alpha_1}
&&\bullet\ar[ur]\ar@{{}*{\cdot}{}}[rrrrrrrr]&&&&&&&&\bullet\ar[ur]&&\star&\cdots\\
\cdots&0&&1&&&&p-2&&p-1&&p&&p+1&&&&&&&&2p&&2p+1&\cdots
}}
$$
where the set $\mathcal{C}$ is stable under $\tau^{2p+1}$. Choose a section $A_{2p+1}'$ in $\prescript{}{s}{\Gamma}_{A}$ as follow:

$$
\vcenter{
\xymatrix {
\bullet\ar[r]^{\alpha_1}&\bullet\ar[r]^{\alpha_2}&\bullet\ar@{{}*{\cdot}{}}[rr]&&\bullet\ar[r]^{\alpha_p}&\star&
\bullet\ar[l]_{\alpha_{p+1}}&&\bullet\ar@{{}*{\cdot}{}}[ll]&\bullet\ar[l]_{\alpha_{2p}}&\bullet\ar[l]_{\alpha_{2p+1}}
}}
$$

By the position of $\mathcal{C}$ in $\prescript{}{s}{\Gamma}_{A}$, one can show that each arrow in $A_{2p+1}'$ corresponds to an irreducible monomorphism. Let $\phi': k(\prescript{}{s}{\Gamma}_{A})\rightarrow k(\prescript{}{s}{\Gamma}_{A})$ be an isomorphism which maps each object in $k(\prescript{}{s}{\Gamma}_{A})$ to itself.  To lift $\phi': k(\prescript{}{s}{\Gamma}_{A})\rightarrow k(\prescript{}{s}{\Gamma}_{A})$ to a functor $\Phi:k(\Gamma_{A})\rightarrow k(\Gamma_{A})$, one can first choose morphisms $\Phi(\alpha_{1}),\cdots,\Phi(\alpha_{2p})$ which lift $\phi'(\alpha_{1}),\cdots,\phi'(\alpha_{2p})$ respectively. Using Lemma \ref{Lem, lift a mesh relation}(1), one can lift arrows in $\prescript{}{s}{\Gamma}_{A}$ from the section $A_{2p+1}'$ to the left. Now assume that the values of $\Phi$ on all arrows of $\prescript{}{s}{\Gamma}_{A}$ except $\beta_{1},\cdots,\beta_{2p}$ have been defined, which satisfy $\Phi(m_x)=0$ for each vertex $x$ such that $x$ is not in $\mathcal{C}$ and such that the values of $\Phi$ on all arrows in $m_x$ have been defined. Since $\alpha_{1},\cdots,\alpha_{2p}$ correspond to irreducible monomorphisms, by Lemma \ref{Lem, lift a mesh relations}(2), one can define $\Phi(\beta_{1}),\cdots,\Phi(\beta_{2p})$ from both sides to the middle such that $\Phi(m_x)=0$ for each vertex $x$ which is not in $\mathcal{C}$. Finally, we define the values of $\Phi$ on the arrows of ${\Gamma}_{A}$ which link to projective vertices.

\medskip
{\bf3. Type $(D_n,s,1)$ with $n\geq 4,s\geq 1$.} The algebra $B=\Lambda(D_n,s,1)$ is given by the quiver below with relations

\begin{enumerate}
\item $\alpha^{i}_{1}\alpha^{i}_{2}\cdots\alpha^{i}_{n-2}=\beta^{i}_{1}\beta^{i}_{0}=\gamma^{i}_{1}\gamma^{i}_{0}$ for all $i\in \{0,\cdots,s-1\}$;
\item For all $i\in \{0,\cdots,s-1\}=\mathbb{Z}/\langle s\rangle$, $\beta^{i+1}_{0}\alpha^{i}_{1}=
\gamma^{i+1}_{0}\alpha^{i}_{1}=\alpha^{i+1}_{n-2}\beta^{i}_{1}=\gamma^{i+1}_{0}\beta^{i}_{1}=\alpha^{i+1}_{n-2}\gamma^{i}_{1}=
\beta^{i+1}_{0}\gamma^{i}_{1}=0$;
\item For all $i\in \{0,\cdots,s-1\}=\mathbb{Z}/\langle s\rangle$ and for all
$j\in \{1,\cdots,n-2\}=\mathbb{Z}/\langle n-2\rangle$, $\alpha^{i+1}_{j-n+2}\cdots\alpha^{i}_{j}=0$, $\beta^{i+1}_{0}\beta^{i}_{1}\beta^{i}_{0}=\beta^{i+1}_{1}\beta^{i+1}_{0}\beta^{i}_{1}=0$, $\gamma^{i+1}_{0}\gamma^{i}_{1}\gamma^{i}_{0}=\gamma^{i+1}_{1}\gamma^{i+1}_{0}\gamma^{i}_{1}=0$.
\end{enumerate}

$$
\vcenter{
\xymatrix@R=1.5pc@C=1.5pc{
&&\bullet\ar[dd]_{\alpha^{s-1}_{1}}&\cdots\ar[l]^{\alpha^{s-1}_{2}}&&& \\
&&&\cdots\ar[ld]_(0.3){\beta^{s-1}_{1}}&\ar@{{}*{\cdot}{}}[dr]&& \\
\bullet\ar[d]_{\alpha^{0}_{n-3}}&&\bullet\ar[ll]_{\alpha^{0}_{n-2}}\ar[ld]_{\beta^{0}_{0}}\ar[d]^{\gamma^{0}_{0}}
&\cdots\ar[l]_(0.3){\gamma^{s-1}_{1}}&&& \\
\vdots\ar[d]_{\alpha^{0}_{2}}&\bullet\ar[rd]_{\beta^{0}_{1}}&\bullet\ar[d]^{\gamma^{0}_{1}}&&\vdots&\vdots&\vdots \\
\bullet\ar[rr]_{\alpha^{0}_{1}} &&\bullet\ar[dd]_{\alpha^{1}_{n-2}}\ar[r]_{\gamma^{1}_{0}}\ar[rd]_{\beta^{1}_{0}}
&\bullet\ar[r]_{\gamma^{1}_{1}}&\bullet\ar[rr]_{\alpha^{2}_{n-2}}\ar[u]^{\gamma^{2}_{0}}\ar[ru]_{\beta^{2}_{0}}&&
\bullet\ar[u]_{\alpha^{2}_{n-3}} \\
&&&\bullet\ar[ru]_{\beta^{1}_{1}}&&& \\
&&\bullet\ar[r]_{\alpha^{1}_{n-3}}&\cdots\ar[r]_{\alpha^{1}_{2}}&\bullet\ar[uu]_{\alpha^{1}_{1}}&&
}}
$$

When $s=1$, we take $A=B$. Then $\prescript{}{s}{\Gamma}_{A}\cong\mathbb{Z}D_{n}/\langle\tau^{2n-3}\rangle$ and we may set
$$\mathcal{C}=\{(0,n-1),(0,n),(n-1,n-2),(1,1),(2,1),\cdots,(n-3,1)\},$$
\noindent where part of $\prescript{}{s}{\Gamma}_{A}\cong\mathbb{Z}D_{n}/\langle\tau^{2n-3}\rangle$ is of the form:

$$
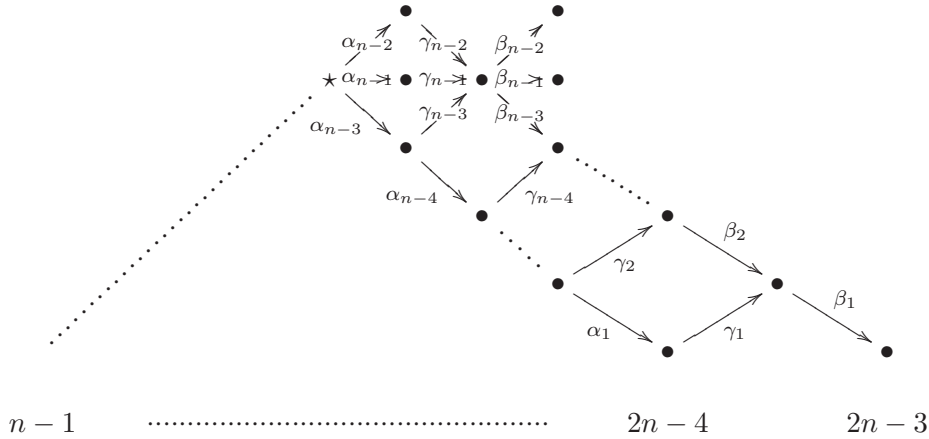
\vcenter{
\xymatrix@R=1.2pc@C=1.4pc {
&&&&&\bullet\ar[dr]|{\gamma_{n-2}}&&\bullet&&& \\
&&&&\star\ar@{{}*{\cdot}{}}[ddddllll]\ar[ur]|{\alpha_{n-2}}\ar[r]|{\alpha_{n-1}}\ar[dr]_{\alpha_{n-3}}&\bullet\ar[r]|{\gamma_{n-1}}
&\bullet\ar[ur]|{\beta_{n-2}}\ar[r]|{\beta_{n-1}}\ar[dr]|{\beta_{n-3}}&\bullet&&& \\
&&&&&\bullet\ar[ur]|{\gamma_{n-3}}\ar[dr]_{\alpha_{n-4}}&&\bullet\ar@{{}*{\cdot}{}}[dr]&&&\\
&&&&&&\bullet\ar@{{}*{\cdot}{}}[dr]\ar[ur]_{\gamma_{n-4}}&&\bullet\ar[dr]^{\beta_{2}}&& \\
&&&&&&&\bullet\ar[ur]_{\gamma_{2}}\ar[dr]_{\alpha_{1}}&&\bullet\ar[dr]^{\beta_{1}}& \\
&&&&&&&&\bullet\ar[ur]_{\gamma_{1}}&&\bullet \\
n-1&\ar@{{}*{\cdot}{}}[rrrrrr]&&&&&&&2n-4&&2n-3
}}
$$

 By the position of $\mathcal{C}$ in $\prescript{}{s}{\Gamma}_{A}$, one can show that the arrows $\beta_{1},\beta_{2},\cdots,\beta_{n-3}$ correspond to irreducible epimorphisms. Choose a section $D_{n}'$ in $\prescript{}{s}{\Gamma}_{A}$ as follow:

$$
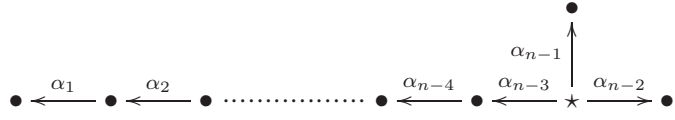
\vcenter{
\xymatrix {
&&&&&&\bullet& \\
\bullet&\bullet\ar[l]_{\alpha_1}&\bullet\ar[l]_{\alpha_2}&&\bullet\ar@{{}*{\cdot}{}}[ll]&\bullet\ar[l]_{\alpha_{n-4}}&
\star\ar[l]_{\alpha_{n-3}}\ar[r]^{\alpha_{n-2}}\ar[u]^{\alpha_{n-1}}&\bullet
}}
$$

Let $\phi': k(\prescript{}{s}{\Gamma}_{A})\rightarrow k(\prescript{}{s}{\Gamma}_{A})$ be an isomorphism which maps each object in $k(\prescript{}{s}{\Gamma}_{A})$ to itself.  To lift $\phi': k(\prescript{}{s}{\Gamma}_{A})\rightarrow k(\prescript{}{s}{\Gamma}_{A})$ to a functor $\Phi:k(\Gamma_{A})\rightarrow k(\Gamma_{A})$, one can first choose morphisms $\Phi(\alpha_{1}),\cdots,\Phi(\alpha_{n-1})$ which lift $\phi'(\alpha_{1}),\cdots,\phi'(\alpha_{n-1})$ respectively. Using Lemma \ref{Lem, lift a mesh relation}(1), one can lift arrows in $\prescript{}{s}{\Gamma}_{A}$ from the section $D_{n}'$ to the left. Now assume that the values of $\Phi$ on all arrows of $\prescript{}{s}{\Gamma}_{A}$ except $\gamma_{1},\cdots,\gamma_{n-1}$ have been defined, which satisfy $\Phi(m_x)=0$ for each vertex $x$ such that $x$ is not in $\mathcal{C}$ and such that the values of $\Phi$ on all arrows in $m_x$ have been defined. Since $\beta_{1},\cdots,\beta_{n-3}$ correspond to irreducible epimorphisms, by Lemma \ref{Lem, lift a mesh relations}(1), one can define $\Phi(\gamma_{1}),\cdots,\Phi(\gamma_{n-1})$ from the bottom to the top such that $\Phi(m_x)=0$ for each vertex $x$ which is not in $\mathcal{C}$. Finally, we define the values of $\Phi$ on the arrows of ${\Gamma}_{A}$ which link to projective vertices.

When $s>1$, since there exists a covering $\mathbb{Z}D_{n}/\langle\tau^{(2n-3)s}\rangle\rightarrow\mathbb{Z}D_{n}/\langle\tau^{2n-3}\rangle$ of stable translation quivers and $$\mathcal{C}=\{(0,n-1),(0,n),(n-1,n-2),(1,1),(2,1),\cdots,(n-3,1)\}$$ is a configuration of $\mathbb{Z}D_{n}/\langle\tau^{2n-3}\rangle$, by \cite[Proposition 2.3]{Riedtmann2}, $$\mathcal{C}'=\{((2n-3)p,n-1),((2n-3)p,n),((2n-3)p+n-1,n-2),((2n-3)p+1,1),$$ $$((2n-3)p+2,1),\cdots,((2n-3)p+n-3,1)\mid 0\leq p\leq s-1\}$$ is a configuration of $\mathbb{Z}D_{n}/\langle\tau^{(2n-3)s}\rangle$. According to \cite[Proposition 1.3]{BLR}, there exists a standard RFS algebra $A$ such that $\Gamma_{A}\cong(\mathbb{Z}D_{n}/\langle\tau^{(2n-3)s}\rangle)_{\mathcal{C}'}$. Using a similar method, it can be shown that each isomorphism $\phi': k(\prescript{}{s}{\Gamma}_{A})\rightarrow k(\prescript{}{s}{\Gamma}_{A})$ which maps each object in $k(\prescript{}{s}{\Gamma}_{A})$ to itself lifts to a functor $\Phi:k(\Gamma_{A})\rightarrow k(\Gamma_{A})$.

\medskip
{\bf 4. Type $(D_{n},s,2)$ with $n\geq 4$, $s\geq 1$.} The stable AR-quiver is of the form $\mathbb{Z}D_{n}/\langle\tau^{(2n-3)s}\eta\rangle$, where $\eta$ is the automorphism of $\mathbb{Z}D_{n}$ defined by the swap of the two high vertices. We may proceed in a similar way as the type $(D_n,s,1)$ with $n\geq 4,s\geq 1$.

\medskip
{\bf 5. Type $(D_{4},s,3)$ with $s\geq 1$.} The stable AR-quiver is of the form $\mathbb{Z}D_{4}/\langle\tau^{5s}\eta\rangle$, where $\eta$ is the automorphism of $\mathbb{Z}D_{4}$ induced from an automorphism of $D_{4}$ of order 3. We may proceed in a similar way as the type $(D_4,s,1)$ with $s\geq 1$.

\medskip
{\bf 6. Type $(D_{3m},s/3,1)$ with $m\geq 2$ and $3\nmid s\geq 1$.} The case $s=1$ has dealt with in Proposition \ref{Prop, lift} (see also Remark \ref{standard-counterpart}). For $s\geq 2$, we can use a similar method as the type $(D_n,s,1)$ with $s\geq 2$. Note that in the case $s=1$ we use the fact that each morphism $(0,3m)\rightarrow(1,3m)$ in $k((\mathbb{Z}D_{3m}/\langle \tau^{2m-1}\rangle)_{\mathcal{C}})$ which factors through a projective vertex is zero. Since there is a covering functor $k((\mathbb{Z}D_{3m}/\langle \tau^{(2m-1)s}\rangle)_{\mathcal{C}'})\rightarrow k((\mathbb{Z}D_{3m}/\langle \tau^{2m-1}\rangle)_{\mathcal{C}})$ which is faithful and sends projective vertices to projective vertices, the similar fact is also true in $k((\mathbb{Z}D_{3m}/\langle \tau^{(2m-1)s}\rangle)_{\mathcal{C}'})$.

\medskip
{\bf 7. Type $(E_{n},s,1)$ with $n\in \{6,7,8\}$ and $s\geq 1$.} The algebra $B=\Lambda(E_{n},s,1)$ is given by the quiver below with relations

\begin{enumerate}
\item $\alpha^{i}_{1}\alpha^{i}_{2}\cdots\alpha^{i}_{n-3}=\beta^{i}_{1}\beta^{i}_{2}\beta^{i}_{3}=\gamma^{i}_{1}\gamma^{i}_{2}$ for all $i\in\{0,\cdots,s-1\}$;
\item For all $i\in \{0,\cdots,s-1\}=\mathbb{Z}/\langle s\rangle$, $\beta^{i+1}_{3}\alpha^{i}_{1}=\gamma^{i+1}_{2}\alpha^{i}_{1}
=\alpha^{i+1}_{n-3}\beta^{i+1}_{1}=\gamma^{i+1}_{2}\beta^{i+1}_{1}=\alpha^{i+1}_{n-3}\gamma^{i}_{1}=\beta^{i+1}_{3}\gamma^{i}_{1}=0$;
\item $\alpha$-paths of length $n-2$ are equal to $0$, $\beta$-paths of length $4$ are equal to $0$, $\gamma$-paths of length $3$ are equal to $0$.
\end{enumerate}

$$
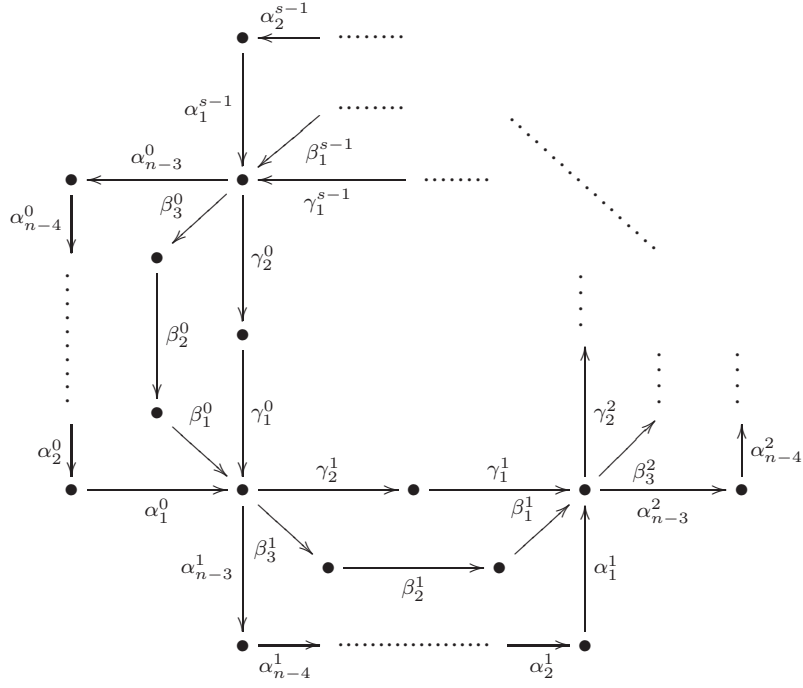
\vcenter{
	\xymatrix@R=1.5pc@C=1.7pc{
    &&\bullet\ar[dd]_{\alpha^{s-1}_{1}}&\ar[l]_{\alpha^{s-1}_{2}}\ar@{{}*{\cdot}{}}[r]&&&&& \\
    &&&\ar[ld]^(0.3){\beta^{s-1}_{1}}\ar@{{}*{\cdot}{}}[r]&&\ar@{{}*{\cdot}{}}[rrdd]&&& \\
    \bullet\ar[d]_{\alpha^{0}_{n-4}}&&\bullet\ar[ll]_{\alpha^{0}_{n-3}}\ar[ld]_(0.6){\beta^{0}_{3}}\ar[dd]^{\gamma^{0}_{2}}&
    &\ar[ll]^{\gamma^{s-1}_{1}}\ar@{{}*{\cdot}{}}[r]&&&& \\
    \ar@{{}*{\cdot}{}}[dd]&\bullet\ar[dd]^{\beta^{0}_{2}}&&&&&\ar@{{}*{\cdot}{}}[d]&& \\
    &&\bullet\ar[dd]^{\gamma^{0}_{1}}&&&&&\ar@{{}*{\cdot}{}}[d]&\ar@{{}*{\cdot}{}}[d] \\
    \ar[d]_{\alpha^{0}_{2}}&\bullet\ar[rd]^(0.3){\beta^{0}_{1}}&&&&&&& \\
    \bullet\ar[rr]_{\alpha^{0}_{1}}&&\bullet\ar[rr]^{\gamma^{1}_{2}}\ar[rd]_{\beta^{1}_{3}}\ar[dd]_{\alpha^{1}_{n-3}}&
    &\bullet\ar[rr]^{\gamma^{1}_{1}}&
    &\bullet\ar[rr]_{\alpha^{2}_{n-3}}\ar[ru]_{\beta^{2}_{3}}\ar[uu]_{\gamma^{2}_{2}}&&\bullet\ar[u]_{\alpha^{2}_{n-4}} \\
    &&&\bullet\ar[rr]_{\beta^{1}_{2}}&&\bullet\ar[ru]^{\beta^{1}_{1}}&&& \\
    &&\bullet\ar[r]_{\alpha^{1}_{n-4}}&\ar@{{}*{\cdot}{}}[rr]&&\ar[r]_{\alpha^{1}_{2}}&\bullet\ar[uu]_{\alpha^{1}_{1}}&&
}}$$

Note that $\prescript{}{s}{\Gamma}_{B}\cong\mathbb{Z}E_{n}/\langle\tau^{m_{n}s}\rangle$, where $m_{6}=11$, $m_{7}=17$, $m_{8}=29$. Similar to the type $(D_n,s,1)$ with $n\geq 4,s\geq 1$, it suffices to consider $s=1$. When $n=6$ and $s=1$,
  we take $A=B$. Then $\mathcal{C}=\{(0,1),(1,1),(-1,6),(-2,5),(-1,5),(5,3)\}$ and part of $\prescript{}{s}{\Gamma}_{A}$ is of the form:

$$
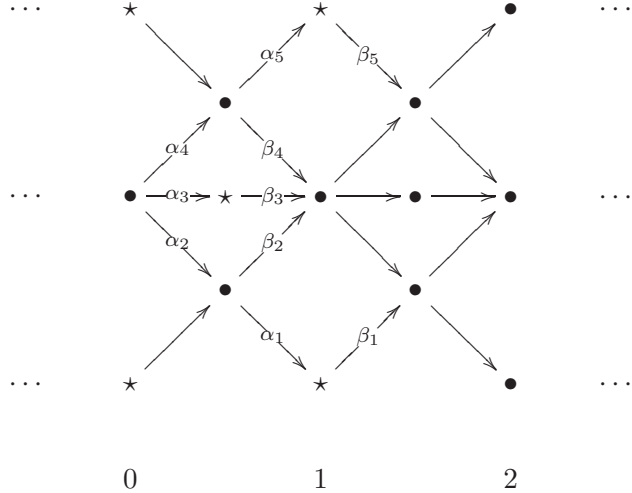
\vcenter{
	\xymatrix {
    \cdots&\star\ar[rd]&&\star\ar[rd]|{\beta_{5}}&&\bullet&\cdots \\
    &&\bullet\ar[ur]|{\alpha_{5}}\ar[dr]|{\beta_{4}}&&\bullet\ar[ur]\ar[rd]&& \\
    \cdots&\bullet\ar[ur]|{\alpha_{4}}\ar[rd]|{\alpha_{2}}\ar[r]|{\alpha_{3}}&\star\ar[r]|{\beta_{3}}&
    \bullet\ar[ur]\ar[rd]\ar[r]&\bullet\ar[r]&\bullet&\cdots \\
    &&\bullet\ar[ur]|{\beta_{2}}\ar[rd]|{\alpha_{1}}&&\bullet\ar[ur]\ar[rd]&& \\
    \cdots&\star\ar[ur]&&\star\ar[ur]|{\beta_{1}}&&\bullet&\cdots \\
    &0&&1&&2&
}}$$

Choose a section $E_{6}'$ in $\prescript{}{s}{\Gamma}_{A}$ as follow:

$$
\vcenter{
	\xymatrix {
	&&(-1,6)&& \\
    (1,1)&(0,2)\ar[l]_{\alpha_{1}}&(-1,3)\ar[l]_{\alpha_{2}}\ar[u]^{\alpha_{3}}\ar[r]^{\alpha_{4}}&(-1,4)\ar[r]^{\alpha_{5}}&(-1,5)
}}$$

By the position of $\mathcal{C}$ in $\prescript{}{s}{\Gamma}_{A}$, $\alpha_{1}$ and $\alpha_{5}$ correspond to irreducible monomorphisms. Let $\phi': k(\prescript{}{s}{\Gamma}_{A})\rightarrow k(\prescript{}{s}{\Gamma}_{A})$ be an isomorphism which maps each object in $k(\prescript{}{s}{\Gamma}_{A})$ to itself. To lift $\phi': k(\prescript{}{s}{\Gamma}_{A})\rightarrow k(\prescript{}{s}{\Gamma}_{A})$ to a functor $\Phi:k(\Gamma_{A})\rightarrow k(\Gamma_{A})$, one can first choose morphisms $\Phi(\alpha_{1}),\cdots,\Phi(\alpha_{5})$ which lift $\phi'(\alpha_{1}),\cdots,\phi'(\alpha_{5})$ respectively. Using Lemma \ref{Lem, lift a mesh relation}(1), one can lift arrows in $\prescript{}{s}{\Gamma}_{A}$ from the section $E_{6}'$ to the left. Now assume that the values of $\Phi$ on all arrows of $\prescript{}{s}{\Gamma}_{A}$ except $\beta_{1},\cdots,\beta_{5}$ have been defined, which satisfy $\Phi(m_x)=0$ for each vertex $x$ such that $x$ is not in $\mathcal{C}$ and such that the values of $\Phi$ on all arrows in $m_x$ have been defined. Using Lemma \ref{Lem, lift a mesh relation}(2), the values $\Phi(\beta_{2}),\Phi(\beta_{3}),\Phi(\beta_{4})$ can be defined such that $\Phi(m_{(-1,3)})=0$. Since $\alpha_{1}$ and $\alpha_{5}$ correspond to irreducible monomorphisms, by Lemma \ref{Lem, lift a mesh relations}(2), one can define $\Phi(\beta_{1})$, $\Phi(\beta_{5})$ such that $\Phi(m_{(0,2)})=0$ and $\Phi(m_{(-1,4)})=0$. Finally, we define the values of $\Phi$ on the arrows of ${\Gamma}_{A}$ which link to  projective vertices.

When $n=7$ or $8$, the proofs are similar to the case $n=6$.

\medskip
{\bf 8. Type $(E_{6},s,2)$ with $s\geq 1$.} The stable AR-quiver is of the form $\mathbb{Z}E_{6}/\langle\tau^{11s}\eta\rangle$, where $\eta$ is the automorphism of $\mathbb{Z}E_{6}$ induced from an automorphism of $E_{6}$ of order 2. We may proceed in a similar way as the type $(E_{6},s,1)$ with $s\geq 1$.

\begin{Rem}\label{}
Proposition \ref{corrected-form} is true for any RFS algebra, the reason is as follows. If $A$ is a RFS algebra of Loewy length $\geq 3$, then every stable auto-equivalence of $A$ is of Morita type, according to Linckelmann's theorem (\cite[Theorem 2.1(iii)]{Linckelmann1996}), Proposition \ref{corrected-form} holds in this case. If $A$ is a RFS algebra of Loewy length $\leq 2$, then every stable auto-equivalence of $A$ which maps each object to itself is the identity functor, which clearly lifts to the identity functor on mod$A$.
\end{Rem}

\section{}

For the benefit of the reader we give a detailed proof of Proposition \ref{H-and-H'} (\cite[Lemma 4.10]{CKL}). First we recall the notion of simple-minded system and the mutation theory of simple-minded systems.

Let $A$ be a self-injective algebra. For $X,Y,Z\in \underline{\mathrm{mod}}A$, $Y$ is called an extension of $X$ and $Z$ if there exists an exact sequence $0\rightarrow X\rightarrow Y\oplus P\rightarrow Z\rightarrow 0$ in $\mathrm{mod}A$, where $P$ is a projective module.

\begin{Def}\label{Def,sms} {\rm(see \cite{KL} or \cite{Dugas2015})}
	Let $A$ be a self-injective $k$-algebra, $\mathcal{S}$ be a set of objects in $\underline{\mathrm{mod}}A$ such that for all $S,T\in \mathcal{S}$, $\underline{Hom}_{A}(S,T)= \left\{\begin{array}{ll} 0 & (S\neq T), \\
	k & (S=T).\end{array}\right.$ Let $\mathcal{F}(\mathcal{S})$ be the smallest subcategory of $\underline{\mathrm{mod}}A$ which contains $\mathcal{S}$ and closed under extensions. $\mathcal{S}$ is called a simple-minded system (sms for short) in $\underline{\mathrm{mod}}A$ if $\mathcal{F}(\mathcal{S})=\underline{\mathrm{mod}}A$.
\end{Def}

By definition, the set of nonprojective simple $A$-modules is an sms in $\underline{\mathrm{mod}}A$.

\begin{Def}\label{Def,mutation} {\rm(\cite[Definition 4.1 and Remark]{Dugas2015})}
	Let $A$ be a self-injective algebra and $\mathcal{S}$ be an sms which is stable under the Nakayama functor $\mathscr{N}=DHom_{A}(-,A)$ up to isomorphisms. Let $\mathcal{X}$ be a subset of $\mathcal{S}$ which is stable under $\mathscr{N}$. The left mutation of the sms $\mathcal{S}$ with respect to $\mathcal{X}$ is the set $\{\mu^{+}_\mathcal{X}(X)\mid X\in \mathcal{S}\}$, where
	\begin{enumerate}
		\item $\mu^{+}_\mathcal{X}(X)=\Omega_{A}^{-1}(X)$, if $X\in \mathcal{X}$;
		\item Otherwise, $\mu^{+}_\mathcal{X}(X)$ is given by the push-out diagram\\
		\xymatrix{
			0\ar[r]  & \Omega_{A}(X)\ar[r]\ar[d] & P\ar[r]\ar[d] & X\ar[r]\ar@{=}[d]  & 0 \\
			0\ar[r] & Y\ar[r] & \mu^{+}_{\mathcal{X}}(X)\ar[r] & X\ar[r] & 0 \\	
		} \\ where $\Omega_{A}(X)\rightarrow Y$ is a minimal left $\mathcal{F}(\mathcal{X})$-approximation of $\Omega_{A}(X)$.
	\end{enumerate}
\end{Def}

It is shown in \cite{Dugas2015} that the left mutation of an sms is again an sms.

\begin{Prop}  {\rm(\cite[Lemma 4.10]{CKL})}
	Let $k$ be an algebraically closed field of characteristic $2$, $\Lambda$ be the representative algebra of nonstandard RFS algebras of type $(D_{3m},1/3,1)$, where $m\geq 2$. Then there exists a standard derived auto-equivalence of $\Lambda$ which induces a stable auto-equivalence $H$ of $\Lambda$ such that $H$ induces the automorphism on $\prescript{}{s}{\Gamma}_{\Lambda}$ by the swap of the two high vertices.
\end{Prop}

\begin{proof} Let $\mathcal{S}_{\Lambda}$ be the set of simple $\Lambda$-modules, $\mathcal{X}=\{2\}\subseteq \mathcal{S}_{\Lambda}$. Since $\Lambda$ is symmetric, $\mathscr{N}\simeq id$ and $\mathcal{X}$ is stable under $\mathscr{N}$. Since $2$ and $2$ have only trivial extension, $\mathcal{F}(\mathcal{X})=add(2)$. The projection ${\rad}P_{1}\rightarrow 2$ is a minimal left $\mathcal{F}(\mathcal{X})$-approximation of ${\rad}P_{1}$. There exists a commutative diagram $$\xymatrix{
		0\ar[r]  & {\rad}P_{1}\ar[r]\ar[d] & P_{1}\ar[r]\ar[d] & 1\ar[r]\ar@{=}[d]  & 0 \\
		0\ar[r] & 2\ar[r] & M\ar[r] & 1\ar[r] & 0, \\	
	}$$
	where $M=$ \unitlength=1.00mm
	\special{em:linewidth 0.4pt} \linethickness{0.4pt}
	\begin{picture}(4.00,7.00)
	\put(0,-1){$2$} \put(0,3){$1$}
	\end{picture}. Then $\mu^{+}_\mathcal{X}(1)=$ \unitlength=1.00mm
	\special{em:linewidth 0.4pt} \linethickness{0.4pt}
	\begin{picture}(4.00,7.00)
	\put(0,-1){$2$} \put(0,3){$1$}
	\end{picture}. For $3\leq i\leq m$, $Hom_{\Lambda}({\rad}P_{i},2)=0$. Then ${\rad}P_{i}\rightarrow 0$ is a minimal left $\mathcal{F}(\mathcal{X})-$approximation of ${\rad}P_{i}$ and $\mu^{+}_\mathcal{X}(i)=i$. Moreover, $$\mu^{+}_\mathcal{X}(2)=\Omega_{\Lambda}^{-1}(2)=\xymatrix@R=0pc{
		2 \\
		3 \\
		\vdots \\
		m \\
		1 \\
		1
	}.$$
In the following proof, we fix the simple $\Lambda$-modules $1,2,\cdots,m$ to the positions $(0,3m)$, $(2m-2,1)$, $\cdots$, $(m,1)$ in the stable AR-quiver $\prescript{}{s}{\Gamma}_{\Lambda}$, respectively (cf. Section 2). Then $\mu^{+}_\mathcal{X}(1)$ corresponds to $(2m-2,3m-1)$, and $\mu^{+}_\mathcal{X}(2)$ corresponds to $(m-1,1)$.
	
By \cite[Okuyama's lemma]{Dugas2015} and noting that the definition of mutation we used here is a variation of Dugas' original one by shifting the objects by ${\Omega_{\Lambda}}^{-1}$, there exist an algebra $\Pi$ and a derived equivalence $F:D^{b}(\mathrm{mod} \Pi)\rightarrow D^{b}(\mathrm{mod} \Lambda)$ which induces a stable equivalence $\phi: \underline{\mathrm{mod}}\Pi\rightarrow \underline{\mathrm{mod}}\Lambda$ sending the set of simple $\Pi$-modules to $\Omega_{\Lambda}(\mu^{+}_\mathcal{X}(\mathcal{S}_{\Lambda}))$. By \cite[Corollary 3.5]{Rickard1991}, we can assume that $F$ is a standard derived equivalence. We may assume that $\Pi$ is basic. Since both $\Omega_{\Lambda}$ and $\tau_{\Lambda}$ lift to derived equivalences, there exists a stable equivalence $H=\tau_{\Lambda}^{-1}{\Omega_{\Lambda}}^{-1}\phi: \underline{\mathrm{mod}}\Pi\rightarrow \underline{\mathrm{mod}}\Lambda$ which lifts to a derived equivalence and sends the set of simple $\Pi$-modules to $\tau_{\Lambda}^{-1}\mu^{+}_\mathcal{X}(\mathcal{S}_{\Lambda})$. Since $\Lambda$ is symmetric and $\Pi$ and $\Lambda$ are derived equivalent, by \cite[Corollary 5.3]{Rickard1991}, $\Pi$ is a symmetric algebra of finite representation type. Hence by Theorem \ref{dclassRFS}, $\Pi$ is nonstandard and $typ(\Pi)=typ(\Lambda)$.

$$\xymatrix{
		\mathcal{S}_{\Lambda}\ar[r]^{mutate\quad\quad}& \Omega_{\Lambda}(\mu^{+}_\mathcal{X}(\mathcal{S}_{\Lambda}))\ar[d]_{\tau_{\Lambda}^{-1}{\Omega_{\Lambda}}^{-1}} &\mathcal{S}_{\Pi}\ar[l]_{\quad\phi}\ar[dl]^{h} &\mathcal{C}_{\Pi}\ar[l]_{\quad\omega_{\Pi}^{-1}}\\
&\tau_{\Lambda}^{-1}\mu^{+}_\mathcal{X}(\mathcal{S}_{\Lambda})=\{(0,3m-1), (m,1), (m+1,1),\cdots ,(2m-2,1)\}\ar[d]_{\eta}&&\\
&\mathcal{S}_{\Lambda}=\{(0,3m), (m,1), (m+1,1),\cdots ,(2m-2,1)\}\ar[d]_{\omega_{\Lambda}}&&\\
&\mathcal{C}_{\Lambda}&&	}$$

Let $\mathcal{C}_{\Pi}$ and $\mathcal{S}_{\Pi}$ be the set of radicals of indecomposable projective $\Pi$-modules and the set of simple $\Pi$-modules respectively. $\Omega_{\Pi}$ induces an automorphism $\omega_{\Pi}$ of $\prescript{}{s}{\Gamma}_{\Pi}$ which sends $\mathcal{S}_{\Pi}$ to $\mathcal{C}_{\Pi}$. Let $h:\prescript{}{s}{\Gamma}_{\Pi}\rightarrow \prescript{}{s}{\Gamma}_{\Lambda}$ be the isomorphism between stable AR-quivers induced by $H$, $\omega_{\Lambda}$ be the automorphism of $\prescript{}{s}{\Gamma}_{\Lambda}$ induced by $\Omega_{\Lambda}$. Since $\tau_{\Lambda}^{-1}\mu^{+}_\mathcal{X}(\mathcal{S}_{\Lambda})$ corresponds to the position $\{(0,3m-1), (m,1), (m+1,1),\cdots ,(2m-2,1)\}$, $\eta h$ sends $\mathcal{S}_{\Pi}$ to $\mathcal{S}_{\Lambda}$ and $\omega_{\Lambda}\eta h{\omega_{\Pi}}^{-1}:\prescript{}{s}{\Gamma}_{\Pi}\rightarrow \prescript{}{s}{\Gamma}_{\Lambda}$ is an isomorphism which maps $\mathcal{C}_{\Pi}$ to $\mathcal{C}_{\Lambda}$, where $\eta$ is the automorphism of $\prescript{}{s}{\Gamma}_{\Lambda}$ which is induced from the automorphism of the quiver $D_{3m}$ by the swap of the two high vertices and $\mathcal{C}_{\Lambda}$ is the set of radicals of indecomposable projective $\Lambda$-modules. Then the AR-quivers of the two nonstandard RFS algebras $\Pi$ and $\Lambda$ are isomorphic. According to Riedtmann's configuration theory (see the paragraph after Definition \ref{combinatorial-configuration} in Section 1), $\Pi$ and $\Lambda$ are isomorphic as algebras. Then $H$ can be identified as a stable auto-equivalence of $\Lambda$ which induces an automorphism $h$ of $\prescript{}{s}{\Gamma}_{\Lambda}$ such that $h$ maps the set of vertices $\{(0,3m), (2m-2,1), \cdots, (m,1)\}$ to $\{(0,3m-1), (2m-2,1), \cdots, (m,1)\}$. Since $Aut(\prescript{}{s}{\Gamma}_{\Lambda})=\langle\tau\rangle\times\langle\eta\rangle$, $h=\eta$ and $H$ induces the automorphism on $\prescript{}{s}{\Gamma}_{\Lambda}$ by the swap of the two high vertices.

\end{proof}

\begin{Rem}\label{general-case-for-D-3m}
The same proof works for all standard RFS algebras of type $(D_{3m},s/3,1)$ with $3\nmid s$ and $m\geq 2$, see \cite[Remark 4.11]{CKL} for an explanation. Combing Corollary A.2, we have proved that every stable auto-equivalence also lifts to a standard derived equivalence in this case.
\end{Rem}

\end{document}